\documentclass[11pt,reqno]{amsart}
\usepackage[all]{xy}
\usepackage{amsfonts}
\usepackage{amssymb}
\usepackage{graphicx}
\graphicspath{{FIGURES/}}
\usepackage[latin1]{inputenc}

\usepackage[pagebackref=true]{hyperref}

\usepackage{amsmath,amssymb,amsthm, latexsym, amscd}
\usepackage{tikz-cd}

\usepackage{hyperref}

\usepackage{mathtools}

\usepackage{color}

\numberwithin{equation}{section}

\newtheorem{theorem}{Theorem}[section]
\newtheorem{proposition}[theorem]{Proposition}
\newtheorem{corollary}[theorem]{Corollary}
\newtheorem{lemma}[theorem]{Lemma}

\theoremstyle{definition}
\newtheorem{definition}[theorem]{Definition}
\newtheorem{example}[theorem]{Example}
\newtheorem{remark}[theorem]{Remark}

\newtheorem{question}[theorem]{Question}

\def\<{\langle}
\def\>{\rangle}

\def\e{\varepsilon}

\def\D{{\Delta}}

\def\R{{\mathbb R}}
\def\Z{{\mathbb Z}}

\def\ent{{\rm ent}}

\def\diam{\mathop{\rm diam}\nolimits}

\def\T{{\mathbb T}}
\def\1{\mathbf 1}

\newcommand{\cf}{{\it cf.}}
\newcommand{\ie}{{\it i.e.}}
\newcommand{\eg}{{\it e.g.}}

\newcommand{\vol}{{\rm vol}}
\newcommand{\length}{{\rm length}}
\newcommand{\eent}{{\rm ent}}

\newcommand{\UW}{{\rm UW}}
\newcommand{\Rad}{{\rm Rad}}

\newcommand{\Cone}{{\rm Cone}}

\long\def\forget#1\forgotten{} %

\begin{document}

\title{Minimal volume entropy of simplicial complexes} 

\author[I.~Babenko]{Ivan Babenko}
\author[S.~Sabourau]{St\'ephane Sabourau}

\thanks{Partially supported by the ANR project Min-Max (ANR-19-CE40-0014).}

\address{Universit\'e Montpellier II, CNRS UMR 5149,
Institut Montpelli\'erain Alexander Grothendieck,
Place Eug\`ene Bataillon, B\^at. 9, CC051, 34095 Montpellier CEDEX 5, France} 

\email{ivan.babenko@umontpellier.fr}

\address{\parbox{\linewidth}{Univ Paris Est Creteil, CNRS, LAMA, F-94010 Creteil, France \\
Univ Gustave Eiffel, LAMA, F-77447 Marne-la-Vall\'ee, France}}


\email{stephane.sabourau@u-pec.fr}

\subjclass[2010]
{Primary 53C23; Secondary 57N65}

\keywords{Minimal volume entropy, simplicial volume, collapsing, exponential growth, algebraic entropy, Urysohn width, Margulis' constant, loop space.}

\begin{abstract}
This article deals with topological assumptions under which the minimal volume entropy of a closed manifold, and more generally of a finite simplicial complex, vanishes or is positive.
In the first part of the article, we present complementing topological conditions expressed in terms of the growth of the fundamental group of the fibers of maps from a given finite simplicial complex to simplicial complexes of lower dimension which ensure that the minimal volume entropy of the simplicial complex either vanishes or is positive.
We also give examples of finite simplicial complexes with zero simplicial volume and arbitrarily large minimal volume entropy.
In the second part of the article, we present topological assumptions related to the exponential growth of certain subgroups in the fundamental group of a finite simplicial complex and to the topology of the loop space of its classifying space under which the minimal volume entropy is positive.
Several examples are presented throughout the text.
\end{abstract}

\maketitle



\section{Introduction}

The notion of volume entropy has attracted a lot of attention since the early works of Efremovich~\cite{Efrem53}, \v{S}varc~\cite{Shvarts55} and Milnor~\cite{Milnor68}.
This Riemannian invariant describes the asymptotic geometry of the universal cover of a Riemannian manifold and is related to the growth of its fundamental group; see~\cite{Shvarts55} and~\cite{Milnor68}.
It is also connected to the dynamics of the geodesic flow.
More specifically, the volume entropy agrees with the topological entropy of the geodesic flow of a closed nonpositively curved manifold and provides a lower bound for it in general; see~\cite{Dinaburg71} and~\cite{Manning79}.
In this article, we study the minimal volume entropy of a closed manifold (and more generally of a finite simplicial complex), a topological invariant introduced by Gromov~\cite{gro82} related to the simplicial volume.
More precisely, we give topological conditions which ensure, in one case, that the minimal volume entropy of a finite simplicial complex is positive and, in the other case, that it vanishes.
Before stating our results, we need to introduce a some definitions.
Unless stated otherwise, all spaces are path-connected.

\medskip

The \emph{volume entropy} of a connected finite simplicial complex~$X$ with a piecewise Riemannian metric~$g$ is the exponential growth rate of the volume of balls in the universal cover of~$X$.
More precisely, it is defined as
\begin{equation}\label{eq:entropy.comp.space}
\ent(X, g) =\lim_{R \to \infty} \frac{1}{R} \log[ \vol \, \widetilde{B}(R)]
\end{equation}
where $\widetilde{B}(R)$ is a ball of radius~$R$ centered at any point in the universal cover of~$X$.
The limit exists and does not depend on the center of the ball. 
Observe that the volume entropy of a finite simplicial complex with a piecewise Riemannian metric is positive if and only if its fundamental group has exponential growth; see Definition~\ref{def:algent}.


The \emph{minimal volume entropy} of a connected finite simplicial $m$-complex~$X$, also known as \emph{asymptotic volume}, see~\cite{B93}, is defined as
\[
\omega(X) = \inf_g \, \ent(X,g) \, \vol(X,g)^\frac{1}{m}
\]
where $g$ runs over the space of all piecewise Riemannian metrics on~$X$.
This topological invariant is known to be an homotopic invariant for closed manifolds~$M$, see~\cite{B93}, and more generally, an invariant depending only on the image of the fundamental class of~$M$ under the classifying map, see~\cite{Brunnbauer08}.
The exact value of the minimal volume entropy (when nontrivial) of a closed manifold is only known in a few cases; see~\cite{katok}, \cite{BCG91}, \cite{Sambusetti99}, \cite{Sambusetti00}, \cite{CB03}, \cite{merlin}.
For instance, the minimal volume entropy of a closed $m$-manifold~$M$ which carries a hyperbolic metric is attained by the hyperbolic metric and is equal to $(m-1) \, \vol(M,{\rm hyp})^\frac{1}{m}$; see~\cite{katok} for $m=2$ and~\cite{BCG91} for~$m \geq 3$.


The \emph{simplicial volume} of a connected closed orientable $m$-manifold~$M$ is defined as
\[
\Vert M \Vert_\Delta = \inf_c \sum_s |r_s|
\]
where the infimum is taken over all real singular $m$-cycles~$c= \sum_s r_s \, \sigma_s$ representing the fundamental class of~$M$ in~$H_m(M;\Z)$.
Here, the maps $\sigma_s:\Delta^m \to M$ are singular $m$-simplices.
The definition extends to finite simplicial $m$-complexes~$X$ whose fundamental class is well-defined, that is, with $H_m(X;\Z)=\Z$.

\medskip

The following inequality of Gromov~\cite[p.~37]{gro82} connects the minimal volume entropy of a connected closed manifold to its simplicial volume (see also~\cite{BK19} for a presentation of this result).
Namely, every connected closed orientable $m$-manifold~$M$ satisfies
\begin{equation} \label{eq:gro}
\omega(M)^m \geq c_m \, \Vert M \Vert_\Delta
\end{equation}
for some positive constant~$c_m$ depending only on~$m$.
Thus, every closed manifold with positive simplicial volume has positive minimal volume entropy.
In particular, the minimal volume entropy of a closed manifold which carries a negatively curved metric is positive; see~\cite{gro82}.
Other topological conditions ensuring the positivity of the minimal volume entropy have recently been obtained by the second author; see~\cite{sab17}.
These conditions are related to the topology of the loop space of the manifold and will be generalized in this article.
In a different direction, the minimal volume entropy provides a lower bound both on the minimal volume, see~\cite{gro82}, and on the systolic volume of a closed manifold, see~\cite{sab} and~\cite{Brunnbauer08}.

\medskip

A natural question to ask in view of~\eqref{eq:gro} is whether every closed orientable manifold with zero simplicial volume has zero minimal volume entropy.
This is known to be true in dimension two~\cite{katok} and in dimension three~\cite{pieroni} (see also~\cite{AP03} combined with Perelman's resolution of Thurston's geometrization conjecture), where the minimal volume entropy is proportional to the simplicial volume.
In dimension four, the same holds true but only for closed orientable geometrizable manifolds; see~\cite{SS09}.
The techniques developed in this article allow us to provide a negative answer for finite simplicial complexes; see Theorem~\ref{theo:ex}.
The question for closed orientable manifolds remains open despite recent progress made with the introduction of the volume entropy semi-norm; see~\cite{BS}.
This geometric semi-norm in homology measures the minimal volume entropy of a real homology class throughout a stabilization process.
Namely, given a path-connected topological space~$X$, it is defined for every ${\bf a} \in H_m(X;\Z)$ as
\begin{equation} \label{eq:defE}
\Vert {\bf a} \Vert_E = \lim_{k \to \infty} \frac{\omega(k \, {\bf a})^m}{k}
\end{equation}
where $\omega({\bf a})$ is the infimum of the minimal relative volume entropy of the maps $f:M \to X$ from an orientable connected closed $m$-pseudomanifold~$M$ to~$X$ such that $f_*([M])={\bf a}$; see~\cite{BS} for a more precise definition.
The volume entropy semi-norm shares similar functorial features with the simplicial volume semi-norm.
Moreover, the two semi-norms are equivalent in every dimension.
That is,
\begin{equation} \label{eq:BS}
c_m \,  \Vert {\bf a} \Vert_\Delta \leq  \Vert {\bf a} \Vert_E \leq C_m \, \Vert {\bf a} \Vert_\Delta
\end{equation}
for some positive constants $c_m$ and~$C_m$ depending only on~$m$.
Thus, a closed manifold with zero simplicial volume has zero volume entropy semi-norm.
See~\cite{BS} for further details.

\medskip

More generally, one may ask for a topological characterization of closed manifolds or simplicial complexes with positive minimal volume entropy.
Such a topological characterization holds for the systolic volume, a topological invariant sharing similar properties with the minimal volume entropy; see~\cite{B93}, \cite{B06}, \cite{B08}, \cite{Brunnbauer08}.
Namely, a closed manifold or simplicial complex has positive systolic volume if and only if it is essential (\cf~Definition~\ref{def:essential}); see~\cite{gro83} and~\cite{B93}.
Though this condition is necessary to ensure that a closed manifold or simplicial complex has positive minimal volume entropy, see~\cite{B93}, it is not sufficient.
Therefore, one should look for stronger or extra assumptions.

\medskip

In the first part of the article, we present topological conditions in this direction.
The first one implies that the minimal volume entropy of a given simplicial complex vanishes and the second one ensures it is positive.
Both these conditions are expressed in terms of the exponential growth of the fundamental group of the fibers of maps between a given simplicial complex and simplicial complexes of lower dimension.

\medskip

Recall that a finitely presented group~$G$ has \emph{exponential growth} if the volume entropy of its Cayley graphs is nonzero and has \emph{uniform exponential growth} at least~$h>0$ if the volume entropy of all its Cayley graphs is at least~$h$; see Definition~\ref{def:algent}.


\medskip

For our topological conditions, we consider connected finite simplicial $m$-complexes~$X$ along with simplicial maps $\pi:X \to P$ to finite simplicial complexes~$P$ of dimension at most~$m-1$. 
We denote by $i_*:\pi_1(F_p) \to \pi_1(X)$ the homomorphism induced by the inclusion map $i:F_p \hookrightarrow X$ of a connected component~$F_p$ of a fiber~$\pi^{-1}(p)$ of~$\pi$.

\medskip

The first condition considered for~$X$ is the fiber $\pi_1$-growth collapsing assumption (or collapsing assumption for short).

\medskip

\textbf{Fiber $\pi_1$-growth collapsing assumption.}
Suppose there exists a simplicial map $\pi:X \to P$ to a finite simplicial complex~$P$ of dimension at most~$m-1$ such that for every connected component~$F_p$ of every fiber~$\pi^{-1}(p)$ with $p \in P$, the finitely presented subgroup~$i_*[\pi_1(F_p)]$ of~$\pi_1(X)$ has subexponential growth.

\medskip

This topological condition ensures that the minimal volume entropy of~$X$ vanishes.

\begin{theorem} \label{theo:omega=0}
Every connected finite simplicial $m$-complex~$X$ satisfying the fiber $\pi_1$-growth collapsing assumption has zero minimal volume entropy, that is, 
\[
\omega(X) = 0.
\]
\end{theorem}

Since closed manifolds admitting an $F$-structure satisfy the collapsing assumption, see Definition~\ref{def:F-structure} for the definition of an $F$-structure, we obtain the following result.
Note that this result can be derived from Paternain and Petean's work on the connection between the topological entropy of the geodesic flow and $F$-structures; see~\cite[Theorem~A]{PP03}.

\begin{corollary} \label{coro:F-structure}
Every closed manifold admitting an $F$-structure (of possibly zero rank) has zero minimal volume entropy.
\end{corollary}

The following corollary presents the effect of the collapsing assumption on the simplicial volume.
It is a direct consequence of Theorem~\ref{theo:omega=0} showing that $\omega(M)=0$ and the inequality~\eqref{eq:gro} providing an upper bound on $||M||_\Delta$ in terms of~$\omega(M)$.
An alternate proof which does not rely on Theorem~\ref{theo:omega=0} is given in Section~\ref{sec:zerovol}.

\begin{corollary} \label{coro:zerovol}
Every closed manifold~$M$ satisfying the collapsing assumption has zero simplicial volume, that is,
\[
||M||_\Delta =0.
\]
\end{corollary}

\medskip

The second condition considered for~$X$ is the fiber $\pi_1$-growth non-collapsing assumption (or non-collapsing assumption for short).

\medskip

\textbf{Fiber $\pi_1$-growth non-collapsing assumption.}
Suppose that for every simplicial map $\pi:X \to P$ to a finite simplicial complex~$P$ of dimension~$m-1$, there exists a connected component~$F_{p_0}$ of some fiber~$\pi^{-1}(p_0)$ with $p_0 \in P$ such that the finitely presented subgroup~$i_*[\pi_1(F_{p_0})]$ of~$\pi_1(X)$ has uniform exponential growth at least~$h$ for some $h=h(X)>0$ depending only on~$X$.

\medskip

This topological condition ensures that the minimal volume entropy of~$X$ does not vanish.

\begin{theorem} \label{theo:omega>0}
Every connected finite simplicial $m$-complex~$X$ satisfying the fiber $\pi_1$-growth non-collapsing assumption has positive minimal volume entropy, that is, 
\[
\omega(X) > 0.
\]
\end{theorem}

As an example, we show in Proposition~\ref{prop:hyp} that closed aspherical manifolds whose fundamental group is a non-elementary word hyperbolic group satisfy the non-collapsing assumption.

\medskip

Note that the fibers of the simplicial map $\pi:X \to P$ in the definition of the collapsing and non-collapsing conditions can always be assumed to be connected; see Proposition~\ref{prop:connected}.
We will also give alternative formulations of both the collapsing and non-collapsing assumptions in terms of open covers of the simplicial complex~$X$; see Proposition~\ref{prop:cover}. 

\medskip

Though Theorem~\ref{theo:omega>0} and Theorem~\ref{theo:omega=0} are expressed in almost complementary terms, they do not provide a complete topological characterization of simplicial complexes with positive minimal volume entropy.
As an additional comment, let us mention that the distinction between exponential growth and uniform exponential growth for finitely generated groups is quite subtle.
Examples of finitely generated groups of exponential growth which do not have uniform exponential growth were first constructed by Wilson~\cite{Wilson04}, answering a question of Gromov.
Still, it is an open question whether all \emph{finitely presented} groups of exponential growth have uniform exponential growth.


\medskip



The new techniques developed in this article allow us to investigate the relationship between the minimal volume entropy and the simplicial volume of simplicial complexes whose fundamental class is well-defined.
In view of the lower and upper bounds~\eqref{eq:BS}, one can ask whether there is a complementary inequality to the bound~\eqref{eq:gro}.
Namely, does there exist a positive constant~$C_m$ such that
\[
\omega(M)^m \leq C_m \, \Vert M \Vert_\Delta
\]
for every connected closed orientable $m$-manifold~$M$?
The question also makes sense for every connected finite simplicial $m$-complex~$X$ whose fundamental class is well-defined.
Our next result provides a negative answer in this case.

\begin{theorem} \label{theo:ex}
There exists a sequence of connected finite simplicial complexes~$X_n$ with a well-defined fundamental class such that the simplicial volume of~$X_n$ is bounded and the minimal volume entropy of~$X_n$ tends to infinity.
\end{theorem}

This result follows from the construction of simplicial complexes and can be deduced from a variety of techniques, including the techniques developed in the first and second part of this article.


\medskip

In the second part of the article, we present another set of topological conditions ensuring the positivity of the minimal volume entropy of a finite simplicial complex.
This set of topological condition is of different nature than the fiber $\pi_1$-growth collapsing condition and is related to the topology of the loop space of the classifying space of simplicial complexes.
This alternative approach extends the results of~\cite{sab17} and relies on different techniques than those in the first part of the article.

\medskip

In order to state our next result, we need to introduce some definitions.

\begin{definition} \label{def:essential}
A connected finite simplicial $m$-complex~$X$ is \emph{essential} if the classifying map $\Phi:X \to K(\pi_1(X),1)$ cannot be deformed to a map whose image lies in the $(m-1)$-skeleton of~$K(\pi_1(X),1)$.
For instance, every connected closed aspherical manifold is essential.

A discrete group~$G$ is \emph{thick} if there exists~$\delta>0$ such that the algebraic entropy of every finitely generated subgroup of~$G$ with exponential growth is at least~$\delta$.
For instance, the fundamental group of a negatively curved closed manifold is thick.
Further examples of thick groups are given in Example~\ref{ex:thick}.

The free loop space~$\Lambda(K)$ of a connected simplicial complex~$K$ is \emph{$m$-tamed} if there exists a deformation of~$\Lambda(K)$ taking all the noncontractible free loops representing elements in any finitely generated subgroup~$H \leqslant \pi_1(K)$ with subexponential growth to loops lying in a simplicial subcomplex~$Z_H \subseteq K$ of dimension less than~$m$; see Definition~\ref{def:tamed} for a more precise statement.
\end{definition}



Combining these various notions, we obtain another criterion that guarantees the minimal volume entropy of a simplicial complex is nonzero.

\begin{theorem} \label{theo:3}
Let $X$ be a connected finite simplicial $m$-complex~$X$.
Suppose that the simplicial complex~$X$ is essential, the fundamental group of~$X$ is thick and the loop space of the classifying space~$K(\pi_1(X),1)$ is $m$-tamed.
Then the simplicial complex~$X$ has positive minimal volume entropy, that is, 
\[
\omega(X) > 0.
\]
\end{theorem}

A similar result has been established in~\cite[Corollary~4.6]{sab17} for essential closed manifolds~$M$ such that the algebraic entropy of every pair of elements of~$\pi_1(M)$ which do not lie in the same infinite cyclic subgroup is bounded away from zero.
A consequence of this assumption, which is satisfied by closed manifolds admitting a negatively curved Riemannian metric, is that the connected components of the free loop space of the classifying space~$K(\pi_1(M),1)$ are contractible (modulo reparametrization); see~\cite{sab17}.
By isolating the topological properties of the loop space of the classifying space that are really necessary and pushing the argument to the next level, we obtain a more satisfying version in Theorem~\ref{theo:3}, which applies to simplicial complexes and allows their fundamental groups to carry abelian subgroups of rank greater than one, and more generally, nilpotent subgroups.
A more general quantitative version of Theorem~\ref{theo:3} can be found in Theorem~\ref{theo:entvol}. 

\medskip

Along the proof of Theorem~\ref{theo:3}, we establish a curvature-free inequality between the Margulis constant defined in Definition~\ref{def:ell} and the volume of an essential simplicial complex whose loop space of the classifying space is tamed; see Corollary~\ref{coro:VL}.

\medskip

Both Theorem~\ref{theo:omega>0} and Theorem~\ref{theo:3} assert that, under some topological conditions, if the volume of a simplicial complex is small then its volume entropy is large.
The conclusion holds true without assuming that the volume of the whole simplicial complex is small.
More specifically, under the same topological conditions, if the volume of every ball of small radius is small, then the volume entropy is large; see Theorem~\ref{theo:B} and Theorem~\ref{theo:entvol} for precise statements.

\medskip

We emphasize that the results in this article, namely Theorem~\ref{theo:omega=0}, Theorem~\ref{theo:omega>0} and Theorem~\ref{theo:3}, hold for the class of finite simplicial complexes and not solely for closed manifolds.
This contrasts with all previous works, which focused on closed manifolds.
In particular, the topological conditions ensuring the positivity of the minimal volume entropy, see Theorem~\ref{theo:omega>0} and Theorem~\ref{theo:3}, apply to simplicial complexes for which the simplicial volume is zero and the inequality~\eqref{eq:gro} does not readily extend.
This is exemplified by Theorem~\ref{theo:ex}. \\

\emph{Acknowledgment.} The second author would like to thank the Fields Institute and the Department of Mathematics at the University of Toronto for their hospitality while this work was completed. 
We express our gratitude to Rostislav Grigorchuk for multiple stimulating discussions.

\section{Simplicial complexes with zero minimal volume entropy}


In this section, we show that the minimal volume entropy and the simplicial volume of a finite simplicial complex satisfying the fiber $\pi_1$-growth collapsing assumption vanish.
We also give a characterization of the collapsing assumption in terms of open covers.


\subsection{Algebraic entropy of groups}
\mbox { }

\medskip

Let us introduce some classical definitions.

\begin{definition} \label{def:algent}
Let $G$ be a finitely generated group and $S$ be a finite symmetric generating set.
Denote by~$\mathcal{G}_S$ the Cayley graph of~$(G,S)$ endowed with the word distance~$d_S$ induced by~$S$, where every edge of~$\mathcal{G}_S$ is of length~$1$.
The \emph{algebraic entropy of~$G$ with respect to~$S$} is defined as
\[
\ent(G,S) = \ent (\mathcal{G}_S, d_S). 
\]
The group~$G$ has \emph{exponential growth} if $\ent(G,S)$ is positive for some (and so any) finite generating set~$S$.
It has \emph{subexponential growth} otherwise.
The \emph{algebraic entropy of~$G$} is defined as
\[
\ent(G) = \inf_S \ent(G,S)
\]
where $S$ runs over all finite symmetric generating sets of~$G$.
The group~$G$ has \emph{uniform exponential growth} at least~$h>0$ if $\ent(G) \geq h$.
\end{definition}

We will also need the following definition.

\begin{definition} \label{def:entN}
Let $X$ be a connected finite simplicial complex with a piecewise Riemannian metric and $A$ be a subcomplex of~$X$ with basepoint~$a$.
Consider the number
\[
\mathcal{N}(A \subseteq X; t) = {\rm card} \! \left\{ [\gamma] \in \pi_1(A,a) \mid \gamma \subseteq A \mbox{ and } \length(\gamma) \leq t \right\}
\]
of homotopy classes of loops of~$A$ based at~$a$ of length at most~$t$.
When $A=X$, we simply write~$\mathcal{N}(X;t)$.
\end{definition}

Recall the following classical result.

\begin{proposition} \label{prop:entN}
With the previous notations,
\[
\ent(X) = \lim_{t \to \infty} \frac{1}{t} \log \mathcal{N}(X; t).
\]
\end{proposition}

A proof of this result can be found in~\cite[Lemma~2.3]{sab} for instance.
\subsection{Connected and non-connected fibers}

\mbox { }

\medskip

The following result shows that we can assume that the fibers of the simplicial map $\pi:X \to P$ in the definition of the collapsing and non-collapsing conditions are connected.

\begin{proposition} \label{prop:connected}
Let $\pi:X \to P$ be a simplicial map between two finite simplicial complexes.
Denote by~$k$ the dimension of~$P$.
Then there exists a surjective simplicial map $\bar{\pi}:X \to \bar{P}$ to a finite simplicial complex~$\bar{P}$ of dimension at most~$k$ such that the fibers of $\bar{\pi}:X \to \bar{P}$ agree with the connected components of the fibers of $\pi:X \to P$.
\end{proposition}

\begin{proof}
Without loss of generality, we can assume that the simplicial map $\pi:M \to P$ is onto.
Define $\bar{P} = X / \! \sim$ as the quotient space of~$X$, where $x \sim y$ if $x$ and~$y$ lie in the same connected component of a fiber of~$\pi:X \to P$.
Since the map $\pi:X \to P$ is simplicial, the quotient space~$\bar{P}$ is a simplicial complex of the same dimension as~$P$.
By construction, the map $\pi:X \to P$ factors out through a simplical map $\bar{\pi}:X \to \bar{P}$ whose fibers agree with the connected components of the fibers of~$\pi$.
\end{proof}

\subsection{Collapsing assumption and zero minimal volume entropy}
\mbox { }

\medskip

In this section, we show the following result.

\begin{theorem} \label{theo:A}
Every connected finite simplicial complex~$X$ satisfying the collapsing assumption has zero minimal volume entropy.
\end{theorem}

\begin{proof}
Consider a simplicial map $\pi:X \to P$ to a finite simplicial $(m-1)$-complex~$P$ as in the collapsing assumption.
By Proposition~\ref{prop:connected}, we can assume that the simplicial map $\pi:X \to P$ is onto and that its fibers~$F_p$ are connected.

Let us introduce a couple of definitions.
An edge of~$X$ is said to be \emph{long} if it is sent to an edge of~$P$ by the simplicial map~$\pi$.
It is said to be \emph{short} otherwise (in which case, it is sent to a vertex of~$P$).
Denote also by~$n_e$ the number of edges of~$X$.

Let $h_X$ and $h_P$ be the piecewise flat simplicial metrics on~$X$ and~$P$, where each simplex is isometric to the standard simplex of the same dimension.
The pullback form~$\pi^*(h_P)$ on~$X$ is a nonnegative symmetric bilinear form which is degenerate everywhere.
Define the simplicial metric 
\[
g_t=\pi^*(h_P) + t^2 h_X
\]
on~$X$ with $t>0$.
Clearly,
\[
\lim_{t \to 0} \vol(X,g_t) = 0.
\]

Let us show that the volume entropy of~$(X,g_t)$ is uniformly bounded from above for~$t >0$ small enough.
First, observe that there exists a constant~$C$ (which does not depend on~$t$) such that
\[
\diam(F_p,h_X) < C
\]
for every (connected) fiber~$F_p=\pi^{-1}(p)$ of~$\pi$.
Here, the diameter is measured with respect to the length structure induced by~$h_X$ on~$F_p$.
Hence,
\[
\diam(F_p,g_t) = t \cdot \diam(F_p,h_X) \xrightarrow[t \to 0]{} 0.
\]
Thus, by taking $t$ small enough, we can assume that $\diam(F_p,g_t) < \frac{1}{2}$ for every vertex $p \in P$.

Let us estimate the number of homotopy classes of edge-loops in~$X$ of \mbox{$g_t$-length} at most~$T$.
Every edge-loop~$\gamma$ in~$X$ of $g_t$-length at most~$T$ decomposes as
\[
\gamma = \alpha_1 \cup \beta_1 \cup \alpha_2 \cup \cdots \cup \beta_k
\]
where $\alpha_i$ is a long edge of~$X$ and $\beta_i$ is an edge-path lying in a (connected) fiber $F_i=\pi^{-1}(p_i)$ of~$\pi$ over a vertex~$p_i \in P$, which joins the terminal endpoint of~$\alpha_i$ to the initial endpoint of~$\alpha_{i+1}$.

Fix a basepoint $a_i \in F_i$.
Denote by~$\ell_i$ the $g_t$-length of~$\beta_i$.
Let~$\bar{\beta}_i$ be the loop of~$F_i$ based at~$a_i$ obtained by connecting the endpoints $x_i$ and~$y_i$ of~$\beta_i$ to the basepoint~$a_i$ along two paths of~$F_i$ of $g_t$-length at most $\diam(F_i,g_t) < \frac{1}{2}$.
The number $\mathcal{N}_{{x_i},{y_i}}(F_i \subseteq (X,g_t);\ell_i)$ of homotopy classes in~$X$ of edge-paths in~$F_i$ with endpoints $x_i$ and~$y_i$, and $g_t$-length at most~$\ell_i$ is bounded by the number of homotopy classes in~$X$ of loops in~$F_i$ based at~$a_i$ of $g_t$-length at most $\ell_i + 2 \, \diam(F_i,g_t)$.
Thus, 
\begin{align}
\mathcal{N}_{{x_i},{y_i}}(F_i \subseteq (X,g_t);\ell_i) & \leq \mathcal{N}(F_i \subseteq (X,g_t); \ell_i + 2 \, \diam(F_i,g_t)) \nonumber  \\
 & \leq \mathcal{N} \left( F_i \subseteq (X,h_X); \frac{\ell_i + 1}{t} \right) \label{eq:N1}
\end{align}
since $t^2 h_X \leq g_t$, where we refer to Definition~\ref{def:entN} for the definition of~$\mathcal{N}$.

By assumption, the subgroups $i_*[\pi_1(F_p)] \leqslant \pi_1(X)$ have subexponential growth.
Thus, for every $\varepsilon > 0$, there exists $T_\e > 0$ such that for every $T \geq T_\e$ and every vertex~$p \in P$,
\begin{equation} \label{eq:N2}
\mathcal{N}(F_p \subseteq (X,h_X);T) \leq \exp \left( \varepsilon \, T \right).
\end{equation}
We can choose $\varepsilon$ to depend on~$t$.
Fix $\varepsilon = t^2$.

It follows from~\eqref{eq:N1} and~\eqref{eq:N2} that the number of possible homotopy classes in~$X$ induced by the edge-path~$\beta_i$ is at most
\[
\mathcal{N}_{{x_i},{y_i}}(F_i \subseteq (X,g_t);\ell_i) \leq \exp \left( \varepsilon \frac{\ell_i +1}{t} \right)
\]
where $\ell_i$ is the $g_t$-length of~$\beta_i$.

Now, there are at most $n_e$ possible choices for each long edge~$\alpha_i$.
Hence, the number of homotopy classes of edge-loops in~$X$ of $g_t$-length at most \mbox{$T > T_0$} is bounded by
\[
n_e^k \, \prod_{i=1}^k \exp \left( \varepsilon \, \frac{\ell_i +1}{t} \right) \leq n_e^k \, \exp \left( {\frac{\varepsilon}{t} \, \sum_{i=1}^k \ell_i} \right) \, \exp \left( k \frac{\varepsilon}{t} \right).
\]
Clearly, $\sum_{i=1}^k \ell_i \leq T$.
Moreover, since every edge~$\alpha_i$ is of $g_t$-length at least~$1$, we have $k \leq T$.
Therefore, the number of homotopy classes of edge-loops in~$X$ of $g_t$-length at most~$T$ is bounded by
\[
n_e^T \, \exp \left( 2 \frac{\varepsilon}{t} T \right).
\]
Taking the log, dividing by~$T$ and letting~$T$ go to infinity, we obtain
\[
\eent(X,g_t) \leq \log n_e + 2t.
\]
This shows that $\eent(X,g_t)$ is uniformly bounded for $t$ close to zero as desired.
\end{proof}

Let us recall the definition of an $F$-structure, first introduced by Cheeger-Gromov~\cite{ChG} in a different context.

\begin{definition} \label{def:F-structure}
A closed manifold~$M$ admits an \emph{$F$-structure} if there are a locally finite open cover~$\{U_i\}$ of~$M$, finite normal coverings~$\pi_i:\tilde{U}_i \to U_i$ and effective smooth actions of tori~$\T^{k_i}$ on~$\tilde{U}_i$ which extend the deck transformations such that if $U_i$ and~$U_j$ intersect each other, then $\pi_i^{-1}(U_i \cap U_j)$ and $\pi_j^{-1}(U_i \cap U_j)$ have a common covering space on which the lifting actions of~$\T^{k_i}$ and~$\T^{k_j}$ commute.
\end{definition}

\forget
\begin{definition} \label{def:F-structure}
A closed manifold~$M$ admits an \emph{$F$-structure} if the following conditions are satisfied:
\begin{itemize}
\item there is a locally finite open cover~$\{U_i\}$ of~$M$ with finite normal coverings~$\pi_i:\tilde{U}_i \to U_i$;
\item there is an effective smooth action of a torus~$\T^{k_i}$ on~$\tilde{U}_i$ which extends the deck transformations;
\item if $U_i$ and~$U_j$ intersect each other, then $\pi_i^{-1}(U_i \cap U_j)$ and $\pi_j^{-1}(U_i \cap U_j)$ have a common covering space on which the lifting actions of~$\T^{k_i}$ and~$\T^{k_j}$ commute.
\end{itemize}
\end{definition}
\forgotten

As an application of Theorem~\ref{theo:A}, we derive the following result, which is also a consequence of~\cite[Theorem~A]{PP03}.

\begin{corollary}
Every closed manifold admitting an $F$-structure (of possibly zero rank) has zero minimal volume entropy.
\end{corollary}

\begin{proof}
By the Slice Theorem and its consequences, see~\cite[Appendix~B]{GGK}, we derive the following properties.
The orbits of the $F$-structure of a closed $m$-manifold~$M$ partition the manifold into closed submanifolds covered by tori; see also~\cite{ChG} and~\cite{PP03}.
The trivial orbits form a submanifold of codimension at least one (at least two if the manifold is orientable) and the orbit space is an orbifold of dimension at most~$m-1$.
Since the fibers of the natural projection from~$M$ to the orbit space have almost abelian fundamental groups, the manifold~$M$ satisfies the collapsing assumption and has zero minimal volume entropy by Theorem~\ref{theo:A}.
\end{proof}

\subsection{Collapsing assumption and open covers} \label{sec:zerovol}

\mbox { }

\medskip

In this section, we give a characterization of the collapsing assumption in terms of open covers.
This characterization is similar to the description of the Urysohn width in terms of either fiber diameters or open covers; see Proposition~\ref{prop:UW}.

\begin{definition} \label{def:subexp}
A path-connected open subset~$U$ of a path-connected topological space~$X$ has \emph{subexponential $\pi_1$-growth in~$X$} if the subgroup $\Gamma_U:=i_*[\pi_1(U)]$ of~$\pi_1(X)$ has subexponential growth, where $i:U \hookrightarrow X$ is the inclusion map.
\end{definition}

\begin{proposition} \label{prop:cover}
A connected finite simplicial $m$-complex~$X$ satisfies the collapsing assumption if and only if it admits a cover of multiplicity at most~$m$ by open subsets of subexponential $\pi_1$-growth in~$X$.
\end{proposition}

\begin{proof}
Suppose that $X$ satisfies the collapsing assumption.
Then there exists a simplicial map $\pi:X \to P$ to a finite simplicial complex~$P$ of dimension $k \leq m-1$ such that for every connected component~$F_p$ of every fiber~$\pi^{-1}(p)$, where $p$ is a vertex of~$P$, the subgroup $i_*[\pi_1(F_p)]$ of~$\pi_1(X)$ has subexponential growth.
Since $P$ is a finite simplicial complex of dimension~$k$, the open cover formed of the open stars~$\textrm{st}(p) \subseteq P$ of the vertices~$p$ of~$P$ has multiplicity $k+1 \leq m$.
The connected components of the preimages~$\pi^{-1}(\textrm{st}(p)) \subseteq X$ of these open stars form an open cover of~$X$ with the same multiplicity $k+1 \leq m$ as the previous cover of~$P$.
Furthermore, the open subsets of this open cover of~$X$ strongly retract by deformation onto the connected components~$F_p$ of the fibers~$\pi^{-1}(p)$.
In particular, they have subexponential $\pi_1$-growth in~$X$. 
This prove the first implication. 

\medskip

For the converse implication, let $\{ U_i \}_{i=1,\cdots,s}$ be a cover of~$X$ of multiplicity at most~$m$ by open subsets of subexponential $\pi_1$-growth in~$X$.
Take a partition of unity~$\{ \phi_i \}$ of~$X$, where each function~$\phi_i:X \to [0,1]$ has its support in~$U_i$.
Consider the map $\Phi:X \to \Delta^{s-1}$ defined by
\[
\Phi(x) = (\phi_1(x),\cdots,\phi_s(x))
\]
in the barycentric coordinates of~$\Delta^{s-1}$.
The nerve~$P$ of the cover~$\{ U_i \}$ is a simplicial complex with one vertex~$v_i$  for each open set~$U_i$, where $v_{i_0},\cdots,v_{i_k}$ span a $k$-simplex of~$P$ if and only if the intersection~$\cap_{j=1}^k U_{i_j}$ is nonempty.
By construction, the dimension of the nerve~$P$ is one less than the multiplicity of the cover~$\{ U_i \}$.
That is, $\dim P \leq m-1$.
We identify in a natural way the vertices~$\{ v_i \}$ of~$P$ with the vertices of~$\Delta^{s-1}$.
With this identification, the nerve~$P$ of~$X$ lies in~$\Delta^{s-1}$.
Furthermore, the image of~$\Phi$ lies in~$P$.
By~\cite[\S2.C]{hatcher}, subdividing~$X$ and~$P$ if necessary, we can approximate $\Phi:X \to P$ by a simplicial map $\pi:X \to P$ close to~$\Phi$ for the $C^0$-topology, whose normalized barycentric coordinates $\pi_i:X \to [0,1]$ have their support in~$U_i$.
Thus, every fiber~$\pi^{-1}(p)$ lies in one of the open subsets~$U_i$.
Therefore, for every connected component~$F_p$ of~$\pi^{-1}(p)$, the subgroup~$i_*[\pi_1(F_p)]$ lies in some subgroup~$i_*[\pi_1(U_i)]$.
Since the open subsets~$U_i$ have subexponential $\pi_1$-growth in~$X$, the subgroups~$i_*[\pi_1(F_p)]$ have subexponential growth and the simplicial complex~$X$ satisfies the collapsing assumption as required.
\end{proof}

An illustration of the characterization of the collapsing assumption in terms of open covers is given by the following example.

\begin{example}
For $i=1,2$, let $M_i$ be a connected closed manifold of dimension~$m \geq 3$ with fundamental group~$\pi_1(M_i)$ of subexponential growth.
Let $N$ be a connected closed $n$-manifold embedded both in~$M_1$ and~$M_2$.
Suppose that the embedding $N \subseteq M_i$ induces a $\pi_1$-monomorphism and that its normal fiber $\nu_i(N) \subseteq TM_i$ is trivial for $i=1,2$.
Define the $m$-manifold
\[
X = (M_1 \setminus U_1(N)) \mathop{\cup}_{N \times S^{m-n-1}} (M_2 \setminus U_2(U))
\]
where $U_i(N)$ is a small tubular neighborhood of~$N$ in~$M_i$.
Take a small tubular neighborhood of~$M_i \setminus U_i(N)$ in~$X$ for $i=1,2$.
This yields a cover of~$X$ by open subsets with subexponential $\pi_1$-growth in~$X$ of multiplicity two.
According to Proposition~\ref{prop:cover}, the closed $m$-manifold~$X$ satisfies the collapsing assumption.
Note however that the fundamental group of~$X$ has exponential growth in general.
Furthermore, when $N$ is reduced to a singleton, the manifold~$X$ is the connected sum~$M_1 \sharp M_2$ of~$M_1$ and~$M_1$.
This construction provides a number examples of closed essential manifolds with a fundamental group of exponential growth and zero volume entropy. 
\end{example}

Combining Theorem~\ref{theo:A} and Proposition~\ref{prop:cover}, we immediately derive the following result.

\begin{corollary}\label{coro:collapsing}
Every connected finite simplicial $m$-complex~$X$ which admits a cover of multiplicity at most~$m$ by open subsets of subexponential $\pi_1$-growth in~$X$ has zero minimal volume entropy. 
\end{corollary}

\subsection{Gromov's vanishing simplicial volume theorem}

\mbox { }

\medskip

The characterization of the collapsing assumption in terms of covers given in Proposition~\ref{prop:cover} allows us to draw a parallel with the simplicial volume through Gromov's vanishing simplicial volume theorem.

\begin{definition}
A group~$G$ is \emph{amenable} if it admits a finitely-additive left-invariant probability measure.
A path-connected open subset~$U$ of a path-connected topological space~$X$ is \emph{amenable in~$X$} if $i_*[\pi_1(U)]$ is an amenable subgroup of~$\pi_1(X)$, where \mbox{$i:U \hookrightarrow X$} is the inclusion map.
\end{definition}

Gromov's vanishing simplicial volume theorem can be stated as follows.

\begin{theorem}[\cite{gro82}, see also~\cite{Ivanov80}] \label{theo:vanishing}
Let $M$ be a connected closed $m$-manifold.
Suppose that $M$ admits a cover by amenable open subset of multiplicity at most~$m$.
Then 
\[ 
\Vert M \Vert_{\D} =0.
\]
In particular, the simplicial volume of a connected closed manifold with amenable fundamental group is zero.
\end{theorem}

\begin{remark}
Recall that every finitely generated group with subexponential growth is amenable; see~\cite{AS57} or~\cite[Theorem~6.11.12]{CSC} for instance.
Thus, every open subset $U \subseteq X$ with subexponential $\pi_1$-growth in~$X$, see Definition~\ref{def:subexp}, is amenable in~$X$.
This observation allows us to present an alternate proof of Corollary~\ref{coro:zerovol} which asserts that connected closed $m$-manifolds satisfying the collapsing assumption have zero simplicial volume.
Indeed, by Proposition~\ref{prop:cover}, such a manifold~$M$ admits a cover of multiplicity at most~$m$ by open subsets of subexponential $\pi_1$-growth in~$M$, and so by amenable open subsets.
It follows from Theorem~\ref{theo:vanishing} that $M$ has zero simplicial volume.
\end{remark}

\forget
\begin{question}
Though every finitely generated group with subexponential growth is amenable, the converse is false, even among finitely presented groups.
For instance, the following finitely presented group
\[
G = \langle a,t \mid a^{t^2} = a^2, [[[a,t^{-1}],a],a]=1 \rangle
\]
is amenable and has exponential growth; see~\cite{BV05}.
(Here, $x^y=y^{-1}xy$ is the conjugate of~$x$ under~$y$.)
According to Theorem~\ref{theo:vanishing}, every closed manifold with fundamental group isomorphic to~$G$ has zero simplicial volume.
Is the minimal volume entropy of a simplicial complex with fundamental group isomorphic to~$G$ always zero?
\end{question}
\forgotten

\section{Simplicial complexes with positive minimal volume entropy}

By relying on the notion of Urysohn width, we show that the minimal volume entropy of a connected finite simplicial complex satisfying the fiber $\pi_1$-growth non-collapsing assumption is positive.


\subsection{Urysohn width and volume}
\mbox { }

\medskip

Let us go over the notion of Urysohn width in metric geometry.

\begin{definition} \label{def:UW}
The \emph{Urysohn $q$-width} of a compact metric space~$X$, denoted by~$\UW_q(X)$, is defined as the least real~$\delta >0$ such that there is a continuous map $\pi:X \to P$ from~$X$ to a simplicial $q$-complex~$P$, where all the fibers~$\pi^{-1}(p)$ have diameter at most~$\delta$ in~$X$.
That~is,
\begin{equation} \label{eq:width}
\UW_q(X) = \adjustlimits  \inf_{\pi:X \to P} \sup_{p \in P} \, \diam_X [\pi^{-1}(p)]
\end{equation}
where $\pi:X \to P$ runs over all continuous map from~$X$ to a simplicial $q$-complex~$P$.
Note that the simplicial complex~$P$ may vary with~$\pi$.
For a simplicial $m$-complex~$X$, we will simply write $\UW(X)$ for~$\UW_{m-1}(X)$.
\end{definition}

The Urysohn width can also be interpreted in terms of open covers; see~\cite[Lemma~0.8]{Guth17} for instance.

\begin{proposition} \label{prop:UW}
A compact metric space~$X$ has Urysohn $q$-width less than~$w$ if and only if there is a finite cover of~$X$ of multiplicity~$q+1$ by open subsets of diamter less than~$w$.
\end{proposition}

In the case of simplicial complexes, we can further require extra structural properties on the map~$\pi:X \to P$ in the previous definition.

\begin{proposition} \label{prop:width}
Let $X$ be a finite simplicial complex with a piecewise Riemannian metric.
Subdividing~$X$ if necessary, we can assume that the maps~$\pi:X \to P$ in the definition of the Urysohn width are simplicial and that their fibers are connected.
\end{proposition}

\begin{proof}
Suppose $\UW_q(X) < \delta$.
By Proposition~\ref{prop:UW}, there is a finite open cover~$\mathcal{U} = \{ U_i \}_{i=1,\cdots, s}$ of~$X$ of multiplicity~$q+1$ and diameter less than~$\delta$.
Consider the natural map $\Phi:X \to P \subseteq \Delta^{s-1}$ to the nerve~$P$ of~$\mathcal{U}$ given by a partition of unity of the cover.
As in the proof of Proposition~\ref{prop:cover}, subdividing~$X$ and~$P$, we can approximate $\Phi:X \to P$ by a simplicial map $\pi:X \to P$ close to~$\Phi$ for the $C^0$-topology, whose normalized barycentric coordinates $\pi_i:X \to [0,1]$ have their support in~$U_i$; see~\cite[\S2.C]{hatcher}.
Thus, every fiber~$\pi^{-1}(p)$ lies in one of the open sets~$U_i$.
Therefore, $\diam_X \pi^{-1}(p) < \delta$.
As a result, we can assume that the map $\pi:X \to P$ is simplicial in the definition of the Urysohn width; see~\eqref{eq:width}.
Now, by Proposition~\ref{prop:connected}, we can replace $\pi:X \to P$ with a simplical map $\bar{\pi}:X \to \bar{P}$ to a simplicial complex~$\bar{P}$ of dimension at most~$q$, whose fibers are connected and of diameter less than~$\delta$.
\end{proof}

We will need the following recent result of Liokumovich-Lishak-Nabutovsky-Rotman~\cite{LLNR}, extending a theorem of L.~Guth~\cite{Guth17}.
The proof of this result was later on simplified by P.~Papasoglu~\cite{Pap}; see also~\cite{Nab}.

\begin{theorem}[\cite{Guth17}, \cite{LLNR}, \cite{Pap}, \cite{Nab}] \label{theo:width}
Let $X$ be a finite simplicial $m$-complex with a piecewise Riemannian metric.
Then
\[
\vol(X) \geq C_m \, \UW(X)^m
\]
where $C_m$ is a explicit positive constant depending only on~$m$.
More precisely, if for some $R>0$ every ball~$B(R) \subseteq X$ of radius~$R$ has volume at most~$C_m \, R^m$ then
\[
\UW(X) \leq R.
\]
\end{theorem}

A more general statement involving the lower dimensional widths and the Hausdorff content of balls holds true; see \cite{LLNR}, \cite{Pap}, \cite{Nab}.

\subsection{Diameter and uniform group growth}
\mbox { }

\medskip

Let us present the following classical result relating the diameter and the volume entropy of a space, similar in spirit to the \v{S}varc-Milnor lemma; see~\cite[\S5.16]{Gromov99}.

\begin{proposition}\label{prop:diam-ent}
Let $U$ be a connected open subset in a connected finite simplicial complex~$X$ with a piecewise Riemannian metric.
Then
\[
\diam_X(U) \cdot \ent(X) \geq \frac{1}{2} \, \ent(\Gamma_U)
\]
where $\Gamma_U := i_*[\pi_1(U)]$ is the image of~$\pi_1(U)$ under the group homomorphism induced by the inclusion map $i:U \hookrightarrow X$.
\end{proposition}

\begin{proof}
The proof of this result is classical.
By~\cite[Proposition~3.22]{Gromov99}, there exist finitely many loops~$\gamma_i \subseteq X$ based at some fixed basepoint~$x_0 \in X$ whose homotopy classes in~$X$ form a generating set~$S$ of~$\Gamma_U = i_*[\pi_1(U)]$ with
\[
\length(\gamma_i) < 2 \, \diam_X(U) + \varepsilon
\]
for some arbitrarily small~$\varepsilon >0$.
Clearly, every homotopy class $\alpha \in \Gamma_U$ can be represented by a loop~$\gamma \subseteq X$ based at~$x_0$ of length at most 
\[
(2 \, \diam_X(U) + \varepsilon) \cdot d_S(e,\alpha)
\]
where $d_S$ is the word distance on~$\Gamma_U$ induced by~$S$.
Thus, the number~$\mathcal{N}(X;T)$ of homotopy classes represented by loops based at~$x_0$ of length at most~$T$, see Definition~\ref{def:entN}, satisfies
\[
\mathcal{N}(X;T) \geq  {\rm card} \left\{ \alpha \in\Gamma_U \mid d_S(e,\alpha) \leq \frac{T}{2 \, \diam_X(U) + \varepsilon} \right\}
\]
It follows from Proposition~\ref{prop:entN} that
\[
\eent(X) \geq \frac{1}{2 \, \diam_X(U) + \varepsilon} \, \ent(\Gamma_U,S)
\]
for every $\varepsilon >0$.
Hence the result.
\end{proof}

\subsection{Non-collapsing assumption and minimal volume entropy}
\mbox{ } 

\medskip

We can now prove the following result complementing Theorem~\ref{theo:A}.

\begin{theorem} \label{theo:B}
Every connected finite simplicial $m$-complex~$X$ satisfying the non-collapsing assumption has positive minimal volume entropy.
More generally, given a piecewise Riemannian metric~$g$ on~$X$, if for some $R>0$ every ball~$B(R) \subseteq X$ of radius~$R$ has volume at most~$C_m \, R^m$ then
\[
\ent(X,g) \geq \frac{h(X)}{2R}
\]
where $C_m$ is an explicit positive constant depending only on~$m$.
\end{theorem}

\begin{proof}
Let $g$ be a piecewise Riemannian metric on~$X$.
By Theorem~\ref{theo:width}  and Proposition~\ref{prop:width}, there exists a constant~$C'_m > 0$ and a simplicial map $\pi:X \to P$ to some simplicial $(m-1)$-complex~$P$, where every fiber~$F_p=\pi^{-1}(p)$ is connected, such that 
\begin{equation} \label{eq:diam1}
\diam_X (F_p) < C'_m \, \vol(X,g)^{\frac{1}{m}}.
\end{equation}
The constant $C'_m$ can be taken arbitrarily close to~$C_m^{-\frac{1}{m}}$, where $C_m$ is the constant in Theorem~\ref{theo:width}.
By the non-collapsing assumption, one of the subgroups~$i_*[\pi_1(F_{p_0})]$ has uniform exponential growth at least~$h(X)$.
The point~$p_0 \in P$ has a small enough open neighborhood $B_{p_0} \subseteq  P$ whose preimage $U_{p_0}=\pi^{-1}(B_{p_0})$ is homotopy equivalent to the fiber~$F_{p_0}$ and has diameter close to the one of~$F_{p_0}$ so that 
\begin{equation} \label{eq:diam2}
\diam_X (U_{p_0}) < C'_m \, \vol(X,g)^{\frac{1}{m}}.
\end{equation}
The group~$\Gamma_{p_0}=i_*[\pi_1(U_{p_0})]$, which is isomorphic to~$i_*[\pi_1(F_{p_0})]$, has uniform exponential growth at least~$h(X)$.
It follows from Proposition~\ref{prop:diam-ent} that 
\begin{equation} \label{eq:chain-diam}
\frac{1}{2} \, h(X) \leq \frac{1}{2} \, \ent(\Gamma_{p_0}) \leq \diam_X(U) \cdot \ent(X,g) \leq C'_m \, \ent(X,g) \, \vol(X,g)^{\frac{1}{m}}.
\end{equation}
Hence, the minimal volume entropy of~$X$ is positive and satisfies 
\[
\omega(X) \geq \frac{1}{2 C'_m} \, h(X) > 0. 
\]

Suppose that every ball~$B(R)$ of radius~$R$ has volume at most~$C_m \, R^m$, where $C_m$ is the constant in Theorem~\ref{theo:width}.
By Theorem~\ref{theo:width}, we can replace the right-hand sides in the inequalities~\eqref{eq:diam1} and~\eqref{eq:diam2} by~$R$.
By making the appropriate change in the inequality chain~\eqref{eq:chain-diam}, we derive the desired lower bound for the volume entropy of~$X$.
\end{proof}

The following result provides examples of simplicial complexex satisfying the non-collapsing assumption.

\begin{proposition} \label{prop:hyp}
Let $X$ be a finite aspherical simplicial $m$-complex whose $m$-th cohomology group with real coefficients~$H^m(X;\R)$ is nonzero.
Suppose the fundamental group of~$X$ is a non-elementary word hyperbolic group.
Then $X$ satisfies the non-collapsing assumption.

In particular, every closed aspherical manifold whose fundamental group is a non-elementary word hyperbolic group satisfies the non-collapsing assumption.
\end{proposition}

\begin{proof}
First observe that since $X$ is aspherical, its fundamental group~$\pi_1(X)$ is torsion-free, otherwise there would exist a finite-dimensional aspherical space with a finite fundamental group, which is impossible.
%
%
Let $\pi:X \to P$ be a simplicial map to a finite simplicial $(m-1)$-complex~$P$.
Suppose that all the subgroups $H=i_*[\pi_1(F_p)] \leqslant \pi_1(X)$, where $F_p$ is a connected component of a fiber~$\pi^{-1}(p)$ and $i:F_p \hookrightarrow X$ is the inclusion map, have subexponential growth (and so are amenable).
That is, the simplicial complex~$X$ satisfies the collapsing assumption.
According to the generalization given by~\cite[Theorem~9.2]{Ivanov80} of Gromov's vanishing simplicial volume theorem, see~Theorem~\ref{theo:vanishing}, the canonical homomorphism $\widehat{H}^m(X;\R) \to H^m(X;\R)$ between bounded cohomology and singular cohomology vanishes.
By~\cite{mineyev}, the canonical homomorphism $\widehat{H}^m(X;\R) \to H^m(X;\R)$ is also surjective.
Since $H^m(X;\R)$ is nonzero, this leads to a contradiction.
Therefore, one of the subgroups~$H_0 = i_*[\pi_1(F_{p_0})] \leqslant \pi_1(X)$ has exponential growth and thus is not abelian.
Now, by~\cite{delzant}, given a torsion-free non-elementary hyperbolic group~$G$, there exists an integer~$n_G$ such that for every $x,y \in G$ which do not commute and for every $n \geq n_G$, the subgroup~$\langle x^n,y^n \rangle$ of~$G$ generated by~$x^n$ and~$y^n$ is free of rank~$2$.
In particular, every finitely presented subgroup~$H \leqslant G$ is either abelian or has uniform exponential growth at least~$h(G) = \frac{1}{n_G} \log(3)$.
This implies that the non-abelian subgroup~$H_0 \leqslant \pi_1(X)$ has uniform exponential growth at least~$h$, where $h=h(X)$ is a positive constant depending only on~$X$.
Hence the result.
%
\end{proof}

\begin{question}
Does every closed orientable manifold~$M$ satisfying the non-collapsing assumption have positive simplicial volume?
Otherwise, find examples of closed orientable manifolds with zero simplicial volume satisfying the non-collapsing assumption.
\end{question}



\subsection{Non-collapsing assumption and open covers} \label{sec:positivevol}

\mbox { }

\medskip

In this section, we give a characterization of the non-collapsing assumption in terms of open covers in the same vein as in the collapsing case; see Section~\ref{sec:zerovol}.

\begin{definition} \label{def:exp.cover}
A cover~$\mathcal{U} = \{U_i \}$ of a path-connected topological space~$X$ by path-connected open subsets has \emph{uniform exponential $\pi_1$-growth} at least~$h$ if the subgroup $\Gamma_{U}:=i_*[\pi_1(U)]$ of~$\pi_1(X)$ has uniform exponential growth at least~$h$ for some open subset~$U$ of~$\mathcal{U}$, where \mbox{$i:U \hookrightarrow X$} is the inclusion map.
\end{definition}

The following proposition mirrors Proposition~\ref{prop:cover} with an analoguous proof, which is left to the reader.


\begin{proposition} \label{prop:exp.cover}
A connected finite simplicial $m$-complex~$X$ satisfies the non-collapsing assumption if and only if every open cover of multiplicity at most~$m$ has uniform exponential $\pi_1$-growth at least $h = h(X) > 0$.
\end{proposition}

The following result is the counterpart of Corollary~\ref{coro:collapsing}.

\begin{corollary}\label{coro:non-collapsing}
Every connected finite simplicial $m$-complex~$X$ whose open covers of multiplicity at most~$m$ have uniform exponential $\pi_1$-growth at least $h = h(X)$ has positive minimal volume entropy. 
\end{corollary}

\subsection{Simplicial volume and minimal volume entropy} \label{subsec:ex}

\mbox{ } 
\medskip

In this section, we construct a sequence of connected finite simplicial complexes~$X_n$ with bounded simplicial volume and minimal volume entropy going to infinity, proving Theorem~\ref{theo:ex}.

\medskip

Let $\Sigma$ be a surface of genus~$h \geq 2$ with a disk removed.
Fix an integer $d \geq 3$.
Consider the simplicial complex~$X$ obtained by attaching~$\Sigma$ to~$S^1$ along a map $\partial \Sigma \to S^1$ of degree~$d$.
The fundamental group of~$X$ has the following presentation
\begin{equation} \label{eq:presentation}
\pi_1(X) = \langle x_1, y_1, \dots , x_h, y_h, z \vert \mathop{\prod}\limits_{i=1}^h [x_i, y_i] z^d = e \rangle.
\end{equation}
The space~$X$ admits a ${\rm CAT}(-1)$ metric and therefore is aspherical; see~\cite{bh}.
(This is also a consequence of the fact that $X$ is the cellular $2$-complex associated to the one-relator group presentation~\eqref{eq:presentation} whose relator is not a proper power; see~\cite{DV73}.)
However, one cannot directly apply Proposition~\ref{prop:hyp} to~$X$.
Indeed, a cellular cohomology computation shows that $H^2(X;\R)$ is trivial.
Observe that a similar cellular homology computation shows that 
\[
H_2(X;\Z) = 0.
\]
\forget
The fundamental group of~$X$ has the following presentation
\[
\pi_1(X) = \langle x_1, y_1, \dots , x_h, y_h, z \vert \mathop{\prod}\limits_{i=1}^h [x_i, y_i] z^d = e\rangle.
\]
Note that the cellular $2$-complex associated to this presentation is given by~$X$.
Since $\pi_1(X)$ is a one-relator group whose relator is not a proper power, its presentation $2$-complex~$X$ is aspherical; see~\cite[Theorem~2.1]{DV73}.
In particular, the simplicial complex~$X$ is essential.
A cellular homology computation shows that $H_2(X;\Z)$ is trivial.
\forgotten

To show that $X$ satisfies the non-collapsing assumption, one needs further topological insight.
Consider a simplicial map $\pi: X \to P$ to a one-dimensional simplicial complex (\ie, a graph).
The space~$X$ admits a $d$-sheeted cover $p: \widehat{X} \to X$ with 
\[
\widehat{X} = S^1 \mathop{\cup}\limits_{\psi_1} \Sigma_1 \cdots \mathop{\cup}\limits_{\psi_d} \Sigma_d
\]
where the $d$ copies~$\Sigma_i$ of~$\Sigma$ with $i=1,\dots,d$ are attached to~$S^1$ along the maps $\psi_i : \partial \Sigma_i \longrightarrow S^1$ defined as the rotations of angle~$\frac{2\pi}{d}i$.
The cover~$\widehat{X}$ contains a $\pi_1$-embedded genus~$2h$ surface~$M$.
Consider the composite map $f= \pi \circ p :M \to P$.
Proposition~\ref{prop:hyp} ensures that for some $p\in P$, the subgroup~$i_*[\pi_1(F_p)] \leqslant \pi_1(M)$ has uniform exponential growth at least~$h(\pi_1(M)) = \frac{1}{\log(3)}$. 
Thus, since $\pi_1(M)$ is a subgroup of~$\pi_1(\widehat{X})$, and so of~$\pi_1(X)$, the space~$X$ satisfies the non-collapsing assumption.

By Theorem~\ref{theo:B}, the minimal volume entropy of~$X$ is positive.
Namely,
\begin{equation} \label{eq:omegaX}
\omega(X) \geq \frac{1}{2C'_2 \log(3)} >0
\end{equation}
where~$C'_2$ is the constant involved in~\eqref{eq:diam1}.
The positivity of the minimal volume entropy of~$X$ can also be deduced from the techniques developed in the second part of the article (see Example~\ref{ex:final}) or from simplicial volume considerations (see below).

\medskip

Consider the simplicial complex~$X_n$ defined as the bouquet 
\[
X_n = \T^2 \vee \left( \bigvee_{i=1}^n X \right)
\]
of the torus with $n$ copies of~$X$.
By construction, we have
\[
H_2(X_n;\Z) = H_2(\T^2;\Z) = \Z
\]
and 
\[
\Vert X_n \Vert_\Delta  = 0.
\]
By~\cite[Theorem~2.6]{BS}, we obtain
\[
\omega(X_n)^2 = n \, \omega(X)^2.
\]
Since $\omega(X)$ is positive, see~\eqref{eq:omegaX}, we conclude that $\omega(X_n)$ goes to infinity, which proves Theorem~\ref{theo:ex}.

\begin{remark}
Applying the ``graph of groups" construction, see~\cite[\S1.B]{hatcher}, we can construct further examples of asperical simplicial complexes satisfying the non-collapsing assumption.
\end{remark}

The previous example extends in higher dimension as follows using different arguments.
Remove a ball from a closed manifold of dimension~$m=2k$ with positive simplicial volume.
The resulting space~$\Sigma$ is a manifold with boundary $\partial \Sigma \simeq S^{2k-1}$.
Fix an integer $d \geq 3$.
Denote by~$Y$ the quotient of~$\Sigma$ by the natural free action of~$\Z_d$ on~$S^{2k-1}$.
Observe that $\pi_1(Y) \simeq \pi_1(\Sigma) \ast \Z_d$ and $H_m(Y;\Z)=0$.
Define the simplicial $m$-complex
\[
X_n = \#_{i=1}^n Y_i
\]
by taking the connected sum of $n$ copies of~$Y$.
Note that $H_m(X_n;\Z)=0$.

The space~$X_n$ admits a $d$-sheeted cyclic cover which can be described as follows.
The connected sum $\#_{i=1}^n \Sigma_i$ of $n$ copies of~$\Sigma$ is a manifold whose boundary identifies with the disjoint union~$\sqcup S_i^{2k-1}$ of $n$ spheres.
Let $\widehat{X}_n$ be the space obtained by gluing $d$ copies of~$\#_{i=1}^n \Sigma_i$ along this disjoint union
\[
\widehat{X}_n = ( \sqcup S_i^{2k-1} ) \cup_{\psi_1} ( \#_{i=1}^n \Sigma_i ) \cdots \cup_{\psi_d} ( \#_{i=1}^n \Sigma_i )
\]
where the attaching maps~$\psi_j$ are given by the action of~$\alpha^j$ on the boundary components of~$\#_{i=1}^n \Sigma_i$ (for a fixed generator~$\alpha$ of~$\Z_d$).
The cover $\widehat{X}_n \to X_n$ is the natural map sending the $d$ copies $\#_{i=1}^n \Sigma_i$ to~$X_n$.
By the comparison principle, see~\cite[Lemma~4.1]{Brunnbauer08}, we have 
\begin{equation} \label{eq:d}
\omega(\widehat{X}_n) \leq d^\frac{1}{m} \, \omega(X_n).
\end{equation}

Now, take two copies $\#_{i=1}^n \Sigma_i$ and $\#_{i=1}^n \bar{\Sigma}_i$ in~$\widehat{X}_n$.
By construction, the boundaries~$\partial \Sigma_i$ and~$\partial \bar{\Sigma}_i$ agree and the union
\[
M_n = ( \#_{i=1}^n \Sigma_i ) \cup ( \#_{i=1}^n \bar{\Sigma}_i )
\]
is a closed $m$-manifold homeomorphic to
\[
M_n \simeq \#_{i=1}^n (  \Sigma_i \# \bar{\Sigma}_i ) \, \, \#_{i=1}^n ( S^1 \times S^{2k-1} ).
\]
Since the simplicial volume is additive under connected sums in dimension at least three, see~\cite{gro82}, we obtain
\[
\Vert M_n \Vert_\Delta = 2n \, \Vert \Sigma \Vert_\Delta > 0.
\]
Thus, by~\eqref{eq:gro}, the minimal volume entropy~$\omega(M_n)$ of~$M_n$ goes to infinity when $n$ or~$\Vert \Sigma \Vert_\Delta$ tend to infinity.

To conclude, consider the simplicial $m$-complex~$Z_n$ defined as the connected sum
\[
Z_n= X_n \# \T^m.
\]
Observe that $Z_n$ is a cellular $m$-complex with a single $m$-cell.
Clearly, $H_m(Z_n;\Z)=\Z$ and $\Vert Z_n \Vert_\Delta = 0$. 
Note also that $Z_n$ is not aspherical since its fundamental group has torsion.
By~\cite[Theorem~2.12]{BS} (which still holds when $M_1$, here~$X_n$, is a cellular $m$-complex with a single $m$-cell), we have $\omega(Z_n) \geq \omega(X_n)$.
Since $\pi_1(M_n)$ is a subgroup of~$\pi_1(\widehat{X}_n)$ and the manifold~$M_n$ contained in~$\widehat{X}_n$ has the same dimension~$m$ as~$\widehat{X}_n$, we deduce that $\omega(\widehat{X}_n) \geq \omega(M_n)$.
Thus, by~\eqref{eq:d}, the minimal volume entropy~$\omega(Z_n)$ of~$Z_n$ goes to infinity.

\begin{remark}
Similar examples exist in odd dimensions but their construction is more technical.
\end{remark}

\section{Margulis' constant and volume of simplicial complexes}

In this section, we introduce a generalized Margulis' constant and present topological conditions which ensure that the generalized Margulis' constant provides a curvature-free lower bound on the volume of finite simplicial complexes.

\subsection{Tamed loop space}

\mbox{ }

\medskip

In order to present our criterion, we need to introduce some definitions related to the topology of the manifold and its loop space.

\begin{definition}
Let $K$ be a connected simplicial complex.
For every homotopy class $c \in \pi_1(K)$, let $\hat{c}$ denote the corresponding conjugacy class.
Consider the connected component
\[
\Lambda_{\hat{c}}(K) = \{ \gamma:S^1 \to K \mid \gamma \mbox{ is freely homotopic to a loop representing } \hat{c} \}
\]
of the free loop space~$\Lambda(K)$ of~$K$ containing~$\hat{c}$.
For every subgroup~$H \leqslant \pi_1(K)$, define also
\[
\Lambda_H(K) =\coprod_{c \in H \setminus \{e\}} \Lambda_{\hat{c}}(K).
\]
When $H=\pi_1(K)$, we simply denote this space by $\Lambda_*(K)=\Lambda(K) \setminus \Lambda_{\hat{e}}(K)$, that is, the space of noncontractible loops on~$K$.
\end{definition}

The following notion plays a fundamental role in our approach.

\begin{definition} \label{def:tamed}
The free loop space~$\Lambda(K)$ of a connected simplicial complex~$K$ is \emph{$m$-tamed} if there exists a homotopy 
\begin{equation} \label{eq:flow}
\varphi_t: \Lambda(K) \to \Lambda(K)
\end{equation}
of the identity map on~$\Lambda(K)$ which satisfies the following: for every finitely generated subgroup~$H \leqslant \pi_1(K)$ of subexponential growth, see Definition~\ref{def:algent}, there exists a simplicial subcomplex~$Z_H \subseteq K$ of dimension at most~$m-1$ such that the restriction
\[
\varphi_t: \Lambda_H(K) \to \Lambda_H(K)
\]
is a retraction by deformation of~$\Lambda_H(K)$ onto~$\Lambda_*(Z_H)$, more specifically, onto~$i_*(\Lambda_*(Z_H))$, where $i_*:\Lambda(Z_H) \hookrightarrow \Lambda(K)$ is the map between the free loop spaces induced by the inclusion map $i: Z_H \hookrightarrow K$.

Loosely speaking, this property means that there is a deformation of the free loop space~$\Lambda(K)$ where all the noncontractible loops representing elements in any finitely generated subgroup~$H$ with subexponential growth are deformed into loops lying in a subcomplex~$Z_H \subseteq K$ of dimension at most~$m-1$.
\end{definition}

Examples of manifolds with and without a tamed loop space are presented below.

\begin{example} \label{ex:CAT(-1)}
Every aspherical simplicial $m$-complex~$X$ such that the centralizer of every nontrivial homotopy class in~$\pi_1(X)$ is an infinite cyclic subgroup has an $m$-tamed loop space; see~\cite{hansen} or~\cite[Proposition~3.6]{sab17}.
In particular, every aspherical simplicial $m$-complex with a (torsion-free) Gromov hyperbolic fundamental group has an $m$-tamed loop space.
For instance, every finite simplicial $m$-complex with a ${\rm CAT}(-1)$ metric (\eg, the simplicial complexes constructed in Section~\ref{subsec:ex}) has an $m$-tamed loop space.
\end{example}

\begin{example}
On the other hand, tori~$\T^m$ and real projective spaces~$\R P^n$ do not have $m$-tamed loop spaces.
\end{example}

In order to present other examples, we first need to introduce a definition.

\begin{definition}
By~\cite{rham}, a simply connected complete Riemannian manifold of nonpositive curvature admits a canonical isometric splitting $M=N \times \R^n$, where $N$ cannot be decomposed as a Riemannian product with a Euclidean factor.
The factor~$\R^n$ is called the \emph{Euclidean de Rham factor} of~$M$.
\end{definition}

Further examples of manifolds with a tamed loop space are given by the following proposition.

\begin{proposition} \label{prop:tamed}
Let $M$ be a closed nonpositively curved $m$-manifold whose Euclidean de Rham factor of the universal cover is trivial and whose fundamental group satisfies the Tits alternative\footnote{We can replace the Tits alternative in Proposition~\ref{prop:tamed} by the weaker condition ``every finitely generated subgroup of~$\pi_1(M)$ is virtually solvable or has exponential growth". \label{foot}} (\ie, every finitely generated subgroup of~$\pi_1(M)$ is virtually solvable or contains a free group of rank two).
Then $M$ has an $m$-tamed loop space.
\forget
The following $m$-manifolds have an $m$-tamed loop space:
\begin{enumerate}
\item closed manifolds with nonpositive curvature such as graph manifolds; \label{ex3}
\item manifolds with fundamental group of subexponential growth (e.g., simply connected manifolds, nil-manifolds) times any of these previous manifolds.
\end{enumerate}
\forgotten
\end{proposition}

\begin{proof}
Fix a triangulation on~$M$.
Denote by $\pi: \tilde{M} \to M$ its universal covering.
The fundamental group~$\pi_1(M)$ acts properly and cocompactly by isometries on the universal cover~$\tilde{M}$.
Let $H \leqslant \pi_1(M)$ be a finitely generated subgroup of subexponential growth.
By the Tits alternative and the Solvable Subgroup Theorem~\cite[\S7,8]{bh}, the subgroup~$H$ is virtually abelian.
Since $\pi_1(M)$ is finitely generated, it only contains finitely many subgroups of a given index.
The (finite) intersection of all free abelian subgroups of minimal index in~$H$ is a normal subgroup~$A$ of finite index in~$H$ isometric to~$\Z^k$.
Denote by ${\rm Min}(A)$ the set of points in~$\tilde{M}$ which are moved the minimal distance by every element of~$A$.
By the Flat Torus Theorem~\cite[\S7.1]{bh}, the subset ${\rm Min}(A)$ splits as $C \times \R^k$, where $C$ is a closed convex subset of~$\tilde{M}$.
Since the Euclidean de Rham factor of~$\tilde{M}$ is trivial, the subset $C \times \R^k$ is different from~$\tilde{M}$ and its projection $\pi(C \times \R^k) \subseteq M$ does not cover~$M$.
Therefore, there exists a retraction by deformation~$\theta_t$ of~$\pi(C \times \R^k)$ onto a simplicial subcomplex~$Z_H \subseteq M$ lying in the $(m-1)$-skeleton of~$M$.
Now, by~\cite[Proposition~3.7.1]{bow}, every loop~$\gamma$ in a closed nonpositively curved manifold converges to a closed geodesic~$\gamma_\infty$ under the Birkhoff curve shortening flow.
For a loop~$\gamma$ freely homotopic to a loop representing a nontrivial element of~$H$, the closed geodesic~$\gamma_\infty$ lies in $\pi(C \times \R^k)$.
Applying the deformation~$\theta_t$ to~$\gamma_\infty$, we deform~$\gamma_\infty$ to a loop lying in~$Z_H$.
This gives rise to the desired retraction of~$\Lambda_H(M)$ onto~$\Lambda_*(Z_H)$
\end{proof}

\begin{remark}
The hypothesis of Proposition~\ref{prop:tamed} are stable by taking the product.
More precisely, the product of two closed manifolds satisfying the hypothesis of Proposition~\ref{prop:tamed} also satisfies these hypothesis.
\end{remark}

We immediately deduce the following corollary.

\begin{corollary} \label{coro:neg}
Suppose $M$ is a closed negatively curved $m$-manifold or a closed quotient of an irreducible higher rank symmetric $m$-space of noncompact type.
Then $M$ has an $m$-tamed loop space.
\end{corollary}

\begin{proof}
Let $M$ be a closed negatively curved manifold.
Clearly, the Euclidean de Rham factor of the universal cover of~$M$ is trivial.
By~\cite{BCG11} (see the footnote~\ref{foot} or apply~\cite{BF}), the fundamental group of~$M$ satisfies the Tits alternative.
Thus, we can apply Proposition~\ref{prop:tamed}.

Let $M$ be a closed quotient of an irreducible higher rank symmetric $m$-space of noncompact type.
Every locally symmetric space of noncompact type has nonpositive curvature.
Furthermore, the Euclidean de Rham factor of the universal cover of~$M$ is trivial since it is irreducible of higher rank.
Now, the fundamental group~$\pi_1(M)$ is a finitely generated linear group.
Hence it satisfies the Tits alternative~\cite{tits}.
Thus, we can apply Proposition~\ref{prop:tamed} in this case too.
\end{proof}

We will need the following proposition regarding the topological properties captured by the subcomplex~$Z_H$ in a simplicial complex with a tamed loop space.

\begin{proposition} \label{prop:inj}
Let $K$ be a simplicial complex whose loop space is $m$-tamed.
Let $H \leqslant \pi_1(K)$ be a finitely generated subgroup of subexponential growth.
Then the following properties hold.
\begin{enumerate}
\item The inclusion map $i: Z_H \hookrightarrow K$  is $\pi_1$-injective. \label{isom}
\item Let $H,H' \leqslant \pi_1(K)$ be two finitely generated subgroups of subexponential growth with $H \leqslant H'$. 
Then $Z_H \subseteq Z_{H'}$. \label{Z}
\item If $K$ is aspherical then $Z_H$ is aspherical. \label{asph}
\item Let $K_H$ be the covering of~$K$ with $\pi_1(K_H)=\pi_1(Z_H)$.
Attach the cone~$\Cone(Z_H)$ to~$K_H$ along a lift of~$Z_H$.
If $K$ is aspherical then $K_H \cup \Cone(Z_H)$ is contractible. \label{contractible}
\end{enumerate}
\end{proposition}

\begin{proof}
\mbox{ }

\eqref{isom}.
Let $\gamma_1$ and~$\gamma_2$ be two noncontractible loops of~$Z_H$ which are homotopic in~$K$.
That is, there is a homotopy~$\gamma_s$ of loops of~$K$ between~$\gamma_1$ and~$\gamma_2$.
For every $t$, the deformation~$\varphi_t(\gamma_s)$ is a homotopy from~$\gamma_1$ to~$\gamma_2$.
By assumption on~$\varphi_t$, this homotopy lies in~$Z_H$ for $t=1$.
Thus, $\gamma_1$ and~$\gamma_2$ are homotopic in~$Z_H$ through $\varphi_1(\gamma_s)$.
Therefore, the inclusion map $i: Z_H \hookrightarrow K$ is $\pi_1$-injective.

\eqref{Z}. Let $x \in Z_H$.
Choose a loop~$\gamma$ of~$Z_H$ based at~$x$ representing a nontrivial class of~$H$.
Since $\gamma$ lies in~$Z_H$, the deformation~$\varphi_t:\Lambda_H(K) \to \Lambda_H(K)$ leaves the loop~$\gamma$ fixed.
Since $H \leqslant H'$, the loop~$\gamma$ also lies in~$\Lambda_{H'}(K)$ and the map $\varphi_t: \Lambda_{H'}(K) \to \Lambda_{H'}(K)$ deforms~$\gamma$ to a loop of~$Z_{H'}$.
Thus, $\gamma$ lies in~$Z_{H'}$ so does~$x$.
That is, $Z_H \subseteq Z_{H'}$.

\eqref{asph}.
Let $f:S^i \to Z_H$ with $i \geq 2$.
Denote by $\alpha_u$ the image by~$f$ of the meridian joining the two poles of~$S^i$ and passing through $u \in S^{i-1}$.
Connect the image of these two poles by an arc $\alpha \subseteq Z_H$ such that the loop $\alpha \cup \alpha_u \subseteq Z_H$ is noncontractible in~$Z_H$ (and so in~$K$ by~\eqref{isom}).
Note that the loops $\alpha \cup \alpha_u$ and so the image of~$S^i$ remain fixed under deformation by~$\varphi_t$.
By asphericity of~$K$, there is a homotopy $f_s:S^i \to K$ from~$f$ to a constant map.
Denote by $\alpha_{u,s}$ the image by~$f_s$ of the meridians of~$S^i$.
The endpoints of~$\alpha$ evolve along~$f_s$ and the paths followed by these endpoints, along with~$\alpha$, give rise to a deformation~$\alpha_s$ of~$\alpha$.
The closed curves $\alpha_s \cup \alpha_{u,s}$, with $u \in S^{i-1}$ and $s \in [0,1]$, define a map $S^1 \times B^i \to M$ which takes the soul of $S^1 \times B^i$ to the loop~$\alpha_1$ and whose restriction to~$S^1 \times \partial B^i$ lies in~$Z_H$.
Applying the deformation retraction $\varphi_t:\Lambda_H(K) \to \Lambda_H(K)$, we can assume that the image of the map $S^1 \times B^i \to K$ lies in~$Z_H$.
Using this map, we can construct an extension of $f:S^i \to Z_H$ to~$B^i$, which shows that $Z_H$ is aspherical.

\eqref{contractible} 
By Proposition~\ref{prop:inj}.\eqref{asph}, both $K_H$ and~$Z_H$ are Eilenberg-MacLane spaces with fundamental group~$\pi_1(Z_H)$.
By Whitehead's theorem, see~\cite[Theorem~4.5]{hatcher}, the space~$K_H$ retracts by deformation onto~$Z_H$.
Therefore, the space $K_H \cup \Cone(Z_H)$ obtained by attaching~$\Cone(Z_H)$ to~$K_H$ along a lift of~$Z_H$ is contractible.
\end{proof}

\subsection{Filling invariants}

\mbox{ }

\medskip

Filling invariants were introduced by M.~Gromov in~\cite{gro83} to establish systolic inequalities on essential simplicial complexes.

\begin{definition} \label{def:FR}
Let $X$ be a compact space with a metric~$d$.
The map 
\[
i:(X,d) \hookrightarrow  (C^0(X),|| \cdot ||_\infty)
\]
defined by $i(x)( \cdot ) = d(x, \cdot)$ is an embedding from the metric space~$X$ into the Banach space $C^0(X)$ of continuous functions on~$X$ endowed with the sup-norm~$|| \cdot ||_\infty$.
This natural embedding, also called the Kuratowski embedding, is distance-preserving.
We will consider $X$ isometrically embedded into~$C^0(X)$.

The \emph{$(m-1)$-radius} of the metric space~$X$, denoted by $\Rad_{m-1}(X)$, is the infimum of the positive reals~$\rho$ such that there exists a continuous map
\[
j:X \stackrel{h}{\longrightarrow} P \to C^0(X)
\] 
which factors out through a simplicial $(m-1)$-complex~$P$ such that
\[
d_{C^0}(i(x),j(x)) < \rho
\]
where $i:X \hookrightarrow  C^0(X)$ is the Kuratowski isometric embedding; see~\cite[Appendix~1.(B)]{gro83}.
\end{definition}

By~\cite[Appendix~1.(D)]{gro83}, the $(m-1)$-radius of a finite simplicial $m$-complex~$X$ with a piecewise Riemannian metric agrees with half its $(m-1)$-Urysohn width; see Definition~\ref{def:UW}.
That is,
\[
\Rad_{m-1}(X) = \tfrac{1}{2} \, \UW_{m-1}(X).
\]
By Theorem~\ref{theo:width}, the $(m-1)$-radius of~$X$ yields a lower bound on its volume, providing a positive answer to~\cite[Appendix~1.(C)]{gro83}.

\begin{theorem}[\cite{Guth17}, \cite{LLNR}, \cite{Pap}, \cite{Nab}] \label{theo:GG}
Let $X$ be a connected finite simplicial $m$-complex with a piecewise Riemannian metric.
Then
\[
\vol (X) \geq c_m \, \Rad_{m-1}(X)^m
\]
where $c_m$ is an explicit positive constant depending only on~$m$. 
More precisely, if for some $R>0$ every ball~$B(R) \subseteq X$ of radius~$R$ has volume at most~$c_m \, R^m$, then
\[
\Rad_{m-1}(X) \leq R.
\]
\end{theorem}

\subsection{Margulis' constant for finite simplicial complexes}

\mbox{ }

\medskip

We will need the following definition which is a reminiscence of Margulis' constant.

\begin{definition} \label{def:ell}
Let $X$ be a connected finite simplicial complex with a piecewise Riemannian metric.
Consider a group homomorphism $\phi:\pi_1(X) \to G$ to a discrete group~$G$.
For every $x \in X$ and $\ell >0$, denote by
\[
\Gamma_{\phi,x}^\ell = \langle \phi([\gamma]) \in G \mid \gamma \mbox{ loop of } X \mbox{ based at } x \mbox{ with } \length(\gamma) \leq \ell \rangle
\]
the subgroup of~$G$ generated by the $\phi$-image of the homotopy classes of the loops of~$X$ based at~$x$ of length at most~$\ell$.

Define 
\[
\ell_\phi(X) = \sup \{ \ell \mid \mbox{for every } x \in X, \mbox{the subgroup } \Gamma_{\phi,x}^\ell \leqslant G \mbox{ has a subexponential growth} \}.
\]
With this definition, the subgroup~$\Gamma_{\phi,x}^\ell$ has subexponential growth for every $\ell < \ell_\phi(X)$.

When $\phi$ is the identity homomorphism, we simply denote this metric invariant by~$\ell(X)$.
\end{definition}

We will also need the following notion of essential simplicial complex.

\begin{definition}
Let $X$ be a connected finite simplicial $m$-complex and $\phi:\pi_1(X) \to G$ be a group homomorphism to a discrete group~$G$.
The simplicial $m$-complex~$X$ is \emph{$\phi$-essential} if the classifying map $\Phi:X \to K(G,1)$ induced by~$\phi$ is not homotopic to a continuous map $X \to P \to K(G,1)$ which factors out through a simplicial $(m-1)$-complex~$P$.
When $\phi$ is the identity homomorphism, we simply say that $X$ is essential.
\end{definition}

Essential simplical $2$-complexes can be described by their fundamental group.

\begin{example}
By~\cite[Proposition~3.3]{wall}, see also~\cite[Theorem~12.1]{KRS} for a further characterization, a connected finite simplicial $2$-complex~$X$ is essential if and only if its fundamental group~$\pi_1(X)$ is not free.
\end{example}

We can now state a first bound on the Margulis constant for finite simplicial complexes with a piecewise Riemannian metric.

\begin{theorem} \label{theo:FR}
Let $X$ be a connected finite simplicial $m$-complex and $\phi:\pi_1(X) \to G$ be a group homomorphism to a discrete group~$G$.
Suppose that the simplicial complex~$X$ is $\phi$-essential and that the free loop space of the classifying space~$K(G,1)$ is $m$-tamed.
Then, every piecewise Riemannian metric of~$X$ satisfies
\[
\Rad_{m-1}(X) \geq \frac{1}{10} \, \ell_\phi(X).
\]
\end{theorem}

\begin{proof}
Let $i:X \hookrightarrow C^0(X)$ be the Kuratowski embedding.
Denote by $U_{\rho}(U)$ the  $\rho$-neigborhood of $X \simeq i(X)$ in~$C^0(X)$.
By definition of the $(m-1)$-radius, for every $\rho >  \Rad_{m-1}(X)$, there exists a continous map 
\[
j:X \stackrel{h}{\longrightarrow} P \to U_\rho(X)
\] 
which factors out through a simplicial $(m-1)$-complex~$P$ such that
\[
d_{C^0}(i(x),j(x)) < \rho.
\]
Changing $h$ and~$P$, subdividing~$X$ if necessary, we can assume that the map $h:X \to P$ is simplicial and surjective.

Consider the mapping cylinder~$Q$ of~$h$ defined as
\[
Q = X \times [0,1] / \! \sim
\]
where $(x,t) \sim (x',t')$ if and only if $t=t'=1$ and $h(x)=h(x')$.
Identify $X$ with $\partial Q := X \times \{ 0 \}$ and endow~$Q$ with an $(m+1)$-dimensional simplicial complex structure compatible with~$X$.

Suppose that $\rho < \frac{1}{10} \, \ell_\phi(X)$.
Let $K=K(G,1)$.
We would like to extend the map $\Phi:\partial Q = X \to K$ into a continuous map $F:Q \to K$.
This would imply that the map~$\Phi$ is homotopic to a map which factors out through $h:X \to P$, which would contradict the assumption that $X$ is $\phi$-essential.
Actually, it might not be possible to construct such a map~$F$.
Still, we will show how to adapt the argument to obtain a contradiction. 

\medskip

\noindent \textbf{Extension to the $0$-skeleton.}

\medskip

Let us argue by contradiction. 
Define a continous map 
\[
\sigma:Q \to U_\rho(X) \subseteq C^0(X)
\]
by sending each vertical line $\{ x \} \times [0,1]$ to the segment~$[i(x),j(x)]$.
By construction, the image of~$\sigma$ lies in the $\rho$-tubular neighborhood~$U_\rho(X)$ of~$X$ in~$C^0(X)$.
Denote by~$Q^k$ the $k$-skeleton of~$Q$.
Subdividing $Q$ if necessary, we can assume that the diameter of the images by~$\sigma$ of the simplices of~$Q$ is less than~$\varepsilon > 0$, with $\varepsilon < \frac{1}{5} \, \ell_\phi(X)-2\rho$.
First, we define a map 
\[
f:Q^0 \cup \partial Q \to X
\]
such that $f_{\mid \partial Q}(x)=\sigma(x) = i(x)$ for every $x \in \partial Q=X$ by sending each vertex $v \in Q^0$ to a nearest point from~$\sigma(v)$ in~$X$, as we wish.
Thus, 
\[
d_{C^0}(f(v),\sigma(v)) \leq \rho.
\]

\medskip

\noindent \textbf{Extension to the $1$-skeleton.}

\medskip

Since the inclusion $i:X \hookrightarrow U_\rho(X)$ is distance-preserving, every pair $v$,~$v'$ of adjacent vertices of~$Q$ satisfies
\[
d_{X}(f(v),f(v')) \leq d_{C^0}(f(v),\sigma(v)) + d_{C^0}(\sigma(v),\sigma(v')) + d_{C^0}(\sigma(v'),f(v'))
\]
by the triangle inequality.
Thus,
\[
d_{X}(f(v),f(v')) \leq r
\]
where $r=2 \rho + \varepsilon < \frac{1}{5} \, \ell_\phi(X)$.
We extend the map $f$ to~$Q^1$ by taking every edge~$[v,v']$ of~$Q$ not lying in~$\partial Q=X$ to a minimizing segment in~$X$ joining the images $f(v)$ and~$f(v')$ of their endpoints, as we wish. 
By construction, the boundary~$\partial \Delta^2$ of every $2$-simplex~$\Delta^2$ of~$Q$ is sent by~$f$ to a closed curve of length at most~$3 r$ in~$X$.

\medskip

\noindent \textbf{Extension to the $2$-skeleton.}

\medskip

Consider the composite map
\begin{equation} \label{eq:F1}
F = \Phi \circ f : Q^1 \cup \partial Q \to K.
\end{equation}
whose restriction~$F_{|\partial Q}$ to~$\partial Q = X$ agrees with~$\Phi$.
Let $\Delta=\Delta^2$ be the $2$-simplex of~$Q$.
Though the map~$f$ may not extend to~$\Delta$ in~$K$, we will still define an extension of it to an extension of~$K$.

If the image $F(\partial \Delta)$ of~$\partial \Delta$ is contractible in~$K$, the map $F: Q^1 \cup \partial Q \to K$ clearly extends to~$\Delta$.
Otherwise, we apply the deformation retraction~$\varphi_t$ on the free loop space~$\Lambda(K)$ of~$K$; see~\eqref{eq:flow}.
This flow gives rise to a homotopy from the noncontractible loop~$F(\partial \Delta)$ of~$K$ to a noncontractible closed curve $\gamma_\Delta$ in~$Z_H$, see Definition~\ref{def:tamed}, where $H=H_\Delta$ is the cyclic subgroup of~$G$ generated by the homotopy class of~$\gamma_\Delta$.
This homotopy from~$F(\partial \Delta)$ to~$\gamma_\Delta$ lifts to the covering~$K_H$ of~$K$ corresponding to~$\pi_1(Z_H)$ then contracts to a point within~$\Cone(Z_H)$, giving rise to an extension
\[
F: \Delta \to K_H \cup \Cone(Z_H) \to K \cup \Cone(Z_H)
\]
of $F:\partial \Delta \to K$, where the second map $\Cone(Z_H) \to K \cup \Cone(Z_H)$ is given by the covering $K_H \to K$.
The reason we work in $K_H \cup \Cone(Z_H)$ is because, contrarily to~$K \cup \Cone(Z_H)$, this space is contractible; see Proposition~\ref{prop:inj}.\eqref{contractible}.
This will help us with the subsequent extensions.

Thus, the map $F:Q^1 \cup \partial Q \to K$ extends to a map 
\[
F:Q^2 \cup \partial Q \to \widehat{K}
\]
where 
\[
\widehat{K} = K \cup \left( \bigcup_H \Cone(Z_H) \right)
\]
is the simplicial complex obtained by gluing the $m$-dimensional space~$\Cone(Z_H)$ to~$K$ as above for every finitely generated subgroup~$H \leqslant G$ of subexponential growth.
We emphasize that the union in the definition of~$\widehat{K}$ is over all the finitely generated subgroups of~$G$ of subexponential growth, and not merely cyclic subgroups generated by the homotopic classes~$\gamma_\Delta$ as previously.
This will allow us to carry out our extension argument in higher dimension.
Note that if two such subgroups $H$ and~$H'$ satisfy $H \leqslant H'$ then $Z_H \subseteq Z_{H'}$ and $\Cone(Z_H) \subseteq \Cone(Z_{H'})$; see Proposition~\ref{prop:inj}.\eqref{contractible}.

\medskip

\noindent \textbf{Extension to higher dimensional skeleta.}

\medskip

Let us argue by induction.
Let $F^{(1)}:Q^1 \cup \partial Q \to K$ be the composite map~\eqref{eq:F1} prior to its extension.
For every $k$-simplex $\Delta^k \subseteq Q$, define $H=F^{(1)}_*(\pi_1(\Delta^{k,1}))$, where $\Delta^{k,1}$ is the $1$-skeleton of~$\Delta^k$.
Let us state the induction hypothesis.
Assume that the map $F:Q^{k-1} \cup \partial Q \to \widehat{K}$ extends to a continuous map $F:Q^k \cup \partial Q \to \widehat{K}$ such that the restriction of~$F$ to every $k$-simplex~$\Delta^k$ decomposes as
\[
F:\Delta^k \to K_H \cup \Cone(Z_H) \to K \cup \Cone(Z_H) \hookrightarrow \widehat{K}
\]
where $K_H$ is the covering of~$K$ with fundamental group~$\pi_1(Z_H)$. 
Note that since $H$ is generated by the $\Phi$-image of loops of length less than~$\ell_\phi(X)$, the subgroup~$H \leqslant G$ has subexponential growth.
Furthermore, since the free loop space of~$K$ is $m$-tamed, the subcomplex~$Z_H$ given by Definition~\ref{def:tamed} is of dimension at most~$m-1$.

By construction, the induction hypothesis holds true for $k=2$.

\medskip

Let $\Delta=\Delta^{k+1}$ be a $(k+1)$-simplex of~$Q$.
Denote by~$\Delta^i$ the $k$-faces of~$\Delta$.
Define $H= F^{(1)}_*(\pi_1(\Delta^1))$ and $H_i= F^{(1)}_*(\pi_1(\Delta_i^1))$, where $\Delta^1$ and~$\Delta_i^1$ are the $1$-skeleta of~$\Delta$ and~$\Delta_i$.
By induction, every map $F:\Delta_i \to X$ factors out through
\[
F: \Delta_i \to K_{H_i} \cup \Cone(Z_{H_i}) \to K \cup \Cone(Z_{H_i}) \hookrightarrow \widehat{K}.
\]
Since $H_i \leqslant H$, the covering $K_{H_i} \to K$ factors out through $K_{H_i} \to K_H \to K$.
By Proposition~\ref{prop:inj}.\eqref{Z}, we have $Z_{H_i} \subseteq Z_H$ and so $\Cone(Z_{H_i}) \subseteq \Cone(Z_H)$.
Thus, all the maps $F:\Delta_i \to \widehat{K}$ factor out through the same space $K_H \cup \Cone(Z_H)$ as follows
\[
F:\partial \Delta_i \to K_H \cup \Cone(Z_H) \to K \cup \Cone(Z_H) \hookrightarrow \widehat{K}.
\]
Since $\partial \Delta = \cup_i \Delta_i$ and the space $K_H \cup \Cone(Z_H)$ is contractible, see Proposition~\ref{prop:inj}.\eqref{contractible}, the map 
\[
F:\partial \Delta \to K_H \cup \Cone(Z_H) \to K \cup \Cone(Z_H) \hookrightarrow \widehat{K}
\]
extends to~$\Delta$ as desired.
This yields an extension $F:Q^{k+1} \cup \partial Q \to \widehat{K}$ to the $(k+1)$-skeleton of~$P$, completing the induction argument.

\medskip

In conclusion, there exists a continuous map $F:Q \to \widehat{K}$ which agrees with~$\Phi$ on~$\partial Q=X$.
Thus, the classifying map $\Phi:X \to K \subseteq \widehat{K}$ is homothetic to a simplicial map $X \to P \to \widehat{K}$ factoring out through~$P$.
(Without loss of generality, we can also assume that the map~$\Phi$ is simplicial.)
Since the space~$\widehat{K}$ is obtained by attaching cells of dimension at most~$m$ to~$K$, the homomorphism $H_m(K^m,K^{m-1};\Z) \to H_m(\widehat{K}^m,\widehat{K}^{m-1};\Z)$ induced by the inclusion~$K \subseteq \widehat{K}$ is injective.
Since the map~$\Phi:X \to K \subseteq \widehat{K}$ is homothetic to a degenerate map, the homomorphism $H_m(X,X^{m-1};\Z) \to H_m(K^m,K^{m-1};\Z) \hookrightarrow H_m(\widehat{K}^m,\widehat{K}^{m-1};\Z)$ it induces is trivial.
Therefore, the homomorphism $H_m(X,X^{m-1};\Z) \to H_m(K^m,K^{m-1};\Z)$ induced by $\Phi:X \to K$ is trivial too.
Hence the map~$\Phi$ can be deformed to a map whose image lies in~$K^{m-1}$.
This is absurd since $X$ is $\phi$-essential.
\end{proof}

Theorem~\ref{theo:FR} combined with Theorem~\ref{theo:GG} immediately yields the following corollary.

\begin{corollary} \label{coro:VL}
Let $X$ be a connected finite simplicial $m$-complex and $\phi:\pi_1(X) \to G$ be a group homomorphism to a discrete group~$G$.
Suppose that the free loop space of the classifying space~$K(G,1)$ is $m$-tamed and that the simplicial complex~$X$ is $\phi$-essential.
Then, every piecewise Riemannian metric of~$X$ satisfies
\[
\vol(X) \geq c_m \, \ell_\phi(X)^m
\]
for some positive explicit constant~$c_m$ depending only on~$m$.
More precisely, there exists a ball~$B \subseteq X$ of radius~$\tfrac{1}{10} \, \ell_\phi(X)$ with
\[
\vol (B) \geq c_m \, \ell_\phi(X)^m.
\]
\end{corollary}

\section{Algebraic entropy, volume entropy and volume}

The goal of this section is to establish a uniform lower bound on the minimal volume entropy of finite simplicial complexes satisfying topological assumptions captured in large part by their fundamental group.

\medskip



The following classical comparison result between the algebraic entropy of a finitely generated group and its quotients is straightforward; see Definition~\ref{def:algent}.

\begin{lemma} \label{lem:surj}
Let $\phi:G_1 \to G_2$ be a surjective homomorphism between two finitely generated groups.
Then
\[
\ent(G_1) \geq \ent(G_2).
\]
\end{lemma}

The following classical result relates the volume of a simplicial complex with a piecewise Riemannian metric to the algebraic entropy of the subgroups of its fundamental group.

\begin{proposition} \label{prop:ent-L}
Let $X$ be a connected finite simplicial complex with a piecewise Riemannian metric~$g$.
Let $H \leqslant \pi_1(X)$ be a subgroup generated by the homotopy classes of finitely many loops of length at most~$L$ based at the same point.
Then
\[
\ent(H) \leq \ent(X,g) \cdot L.
\]
\end{proposition}

\begin{proof}
By assumption, the subgroup~$H$ is generated by the homotopy classes $\alpha_1,\cdots,\alpha_k$ of loops of length at most~$L$ based at the same point.
Fix a lift $x \in \widetilde{X}$ of this basepoint.
Every element $\alpha \in H$ can be expressed as a word in the letters $S= \{ \alpha_1^\pm,\cdots,\alpha_k^\pm \}$.
By the triangle inequality, the distance in~$\widetilde{X}$ between $x$ and~$\alpha \cdot x$ satisfies
\[
{\rm dist}_{\widetilde{X}}(x, \alpha \cdot x) \leq d_S(e,\alpha) \cdot L.
\]
Therefore, the ball $B_S(t)$ of radius~$t$ centered at~$e$ in the Cayley graph~$\mathcal{G}_S$ of~$(H,S)$ with respect to the word metric~$d_S$, see Definition~\ref{def:algent}, satisfies
\[
{\rm card} \, B_S(t)  \leq \mathcal{N}(X;L \cdot t)
\]
where $\mathcal{N}(X;t)$ is the number of homotopy classes of pointed loops of~$X$ of length at most~$t$; see Definition~\ref{def:entN}.
Taking the log, dividing by~$L \cdot t$ and letting $t$ go to infinity, we derive from Proposition~\ref{prop:entN} that
\[
\ent(H) \leq \ent(H,S) \leq \ent(X,g) \cdot L.
\]
\end{proof}

We will need the following definition, which can be seen as a quantitative version of the Tits alternative.

\begin{definition}
A discrete group~$G$ is \emph{$\delta$-thick} if the algebraic entropy of every finitely generated subgroup of~$G$ with exponential growth is at least~$\delta$.
A discrete group~$G$ is \emph{thick} if it is $\delta$-thick for some~$\delta >0$.
\end{definition}

\begin{example} \label{ex:thick}
Let $G$ be a finitely generated group.
Then the group~$G$ is $\delta$-thick in any of the following cases:
\begin{itemize}
\item $G$ is a discrete subgroup of the isometry group of an $m$-dimensional Cartan-Hadamard manifold of pinched sectional curvature $-a^2 \leq K \leq -1$ with $\delta=\delta(m,a)$ depending only on~$m$ and~$a$; see~\cite{BCG11}.
Further extensions can be found in~\cite{BCGS} and~\cite{BF}.
\item $G$ acts freely on a $2$-dimensional ${\rm CAT}(0)$ cube complex with $\delta=\frac{1}{10} \log(2)$; see~\cite{KS19}.
Generalizations to groups acting on product of $2$-dimensional ${\rm CAT}(0)$ cube complexes and to groups acting on $m$-dimensional ${\rm CAT}(0)$ cube complexes with isolated flats have been announced; see~\cite{GJN}.
\item $G$ is a group whose $2$-generated subgroups are free with $\delta=\log(3)$. 
Examples of such groups can be found in~\cite{guba}, \cite{bumagin} and~\cite{AO96}.
Generically, all finitely presented groups satisfy this property; see~\cite{AO96}.
\end{itemize}
\end{example}

The following theorem is based on the results of the previous section.

\begin{theorem} \label{theo:entvol}
Let $X$ be a connected finite simplicial $m$-complex and $\phi:\pi_1(X) \to G$ be a group homomorphism.
Suppose that the simplicial complex~$X$ is $\phi$-essential, the group $G$ is $\delta$-thick and the free loop space of the classifying space~$K(G,1)$ is $m$-tamed.
Then
\[
\omega(X) \geq \lambda_m \, \delta
\]
where $\lambda_m$ is an explicit positive constant which depends only on~$m$.
More generally, given a piecewise Riemannian metric~$g$ on~$X$, if for some $R>0$ every ball~$B(R) \subseteq X$ of radius~$R$ has volume at most~$a_m \, R^m$ then
\[
\ent(X,g) \geq \frac{\delta}{R}
\]
where $a_m$ is an explicit positive constant which depends only on~$m$.
\end{theorem}

\begin{proof}
Fix a piecewise Riemannian metric~$g$ on~$X$.
By definition of~$\ell_\phi(X)$, see Definition~\ref{def:ell}, there exists a subgroup~$H \leqslant \pi_1(X)$ generated by finitely many loops of~$X$ of length at most~$\ell_\phi(X)$ based at the same point whose $\phi$-image $H_0=\phi(H) \leqslant G$ has exponential growth.
Since $G$ is $\delta$-thick, the algebraic entropy of~$H_0$ is at least~$\delta$.
By Lemma~\ref{lem:surj}, Proposition~\ref{prop:ent-L} and Corollary~\ref{coro:VL}, we have
\begin{equation} \label{eq:chain}
\delta \leq \ent(H_0) \leq \ent(H) \leq \ent(X,g) \, \ell_\phi(X) \leq c'_m \, \ent(X,g) \, \vol(X,g)^\frac{1}{m}
\end{equation}
for some explicit positive constant~$c'_m$ which depends only on~$m$ and can be taken arbitrarily close to~$c_m^{-\frac{1}{m}}$, where $c_m$ is the constant in Theorem~\ref{theo:GG}.

Let $a_m = \frac{c_m}{10^m}$. 
Suppose that every ball~$B(R)$ of radius~$R$ has volume at most~$a_m \, R^m$.
By Theorem~\ref{theo:FR} and Theorem~\ref{theo:GG}, we derive 
\[
\frac{1}{10} \, \ell_\phi(X) \leq \Rad_{m-1}(X) \leq \frac{1}{10} \, R.
\]
Hence, $\ell_\phi(X) \leq R$.
Replacing this bound in the inequality chain~\eqref{eq:chain}, we obtain the desired lower bound for the volume entropy of~$X$.
\end{proof}

We immediately deduce the following result.

\begin{corollary} \label{coro:final}
Let $X$ be an essential connected finite simplicial $m$-complex whose fundamental group is thick and loop space is $m$-tamed.
Then the minimal volume entropy of~$X$ is positive.
\end{corollary}

\begin{example} \label{ex:final}
Consider the aspherical simplicial $2$-complex~$X$ introduced in Section~\ref{subsec:ex}.
The simplicial complex~$X$ admits a ${\rm CAT}(0)$ cube complex structure induced by the ${\rm CAT}(0)$ cube complex structure of the punctured surface~$\Sigma$ of Section~\ref{subsec:ex}.
It follows from the second item of Example~\ref{ex:thick} that the fundamental group of~$X$ is thick; see~\cite{KS19}.
Now, as previously observed in Example~\ref{ex:CAT(-1)}, the loop space of~$X$ is $2$-tamed.
Therefore, the minimal volume entropy of~$X$ is positive by Corollary~\ref{coro:final}.
\end{example}


\end{document}